\documentclass{amsart}
\usepackage{geometry}
\usepackage{amssymb} %% Yo lo puse
\usepackage[cmtip,all]{xy} %% Yo lo puse

\usepackage{tikz}% Yo lo puse 
\usetikzlibrary{arrows}% Yo lo puse 
 \usetikzlibrary{snakes}% Yo lo puse 
 \usetikzlibrary{shapes}% Yo lo puse 

\newtheorem{theorem}{Theorem}[section]
\newtheorem{lemma}[theorem]{Lemma}

\theoremstyle{definition}
\newtheorem{definition}[theorem]{Definition}

\newtheorem{proposition}[theorem]{Proposition} %%% Yo lo puse
\newtheorem{corollary}[theorem]{Corollary}%% Yo lo puse

\theoremstyle{remark}
\newtheorem{remark}[theorem]{Remark}

\numberwithin{equation}{section}

\begin{document}

\title[Singularity categories]{Homological epimorphisms in functor categories and singularity categories}


%    author one information

\author{Valente Santiago Vargas}
\address{Departamento de Matem\'aticas, Facultad de Ciencias, Universidad Nacional Aut\'onoma de M\'exico,
Circuito Exterior, Ciudad Universitaria, C.P. 04510, Ciudad de M\'exico, MEXICO.}
\curraddr{}
\email{valente.santiago@ciencias.unam.mx}
\thanks{The authors thank to the project PAPIIT-Universidad Nacional Aut\'onoma de M\'exico IN100124.}

%author two information
\author{Juan Andr\'es Orozco Guti\'errez}
\address{Departamento de Matem\'aticas, Facultad de Ciencias, Universidad Nacional Aut\'onoma de M\'exico,
Circuito Exterior, Ciudad Universitaria, C.P. 04510, Ciudad de M\'exico, MEXICO.}
\curraddr{}
\email{juan\_andres@ciencias.unam.mx}
\thanks{}

%    \subjclass is required.

\date{}

\dedicatory{}

%    "Communicated by" -- provide editor's name; required.
\commby{}

%    Abstract is required.
\begin{abstract}

Given a homological epimorphism $\pi:\mathcal{C}\longrightarrow \mathcal{C}/\mathcal{I}$ between $K$-categories, we show that if the ideal $\mathcal{I}$ satisfies certain conditions, then there exists an equivalence between the  singularity categories $\mathbf{D}_{sg}(\mathrm{Mod}(\mathcal{C}))$ and $ \mathbf{D}_{sg}(\mathrm{Mod}(\mathcal{C}/\mathcal{I}))$. This result generalizes the one obtained by Xiao-Wu Chen in \cite{Xiao}. We apply our result to the one point extension category and show that there is a singular equivalence between a $K$-category $\mathcal{U}$ and its one point extension category $\Lambda:=\left[ \begin{smallmatrix}
\mathcal{C}_{K} & 0 \\ 
\underline{M} & \mathcal{U}
\end{smallmatrix}\right].$
 \end{abstract}

\maketitle

\section{Introduction}\label{sec:1}

The singularity category $\mathbf{D}_{sg}(A)$ of an algebra A over a field $K$, introduced by
R.O. Buchweitz in \cite{Buchweitz}, is defined as the Verdier quotient 
$$\mathbf{D}_{sg}(A)=\mathbf{D}^{b}(\mathrm{mod}(A))/\mathrm{perf}(A)$$ of the bounded derived category $\mathbf{D}^{b}(\mathrm{mod}(A))$ by the category of perfect complexes.
In recent years, D. Orlov (\cite{Orlov1} ) rediscovered the notion of singularity categories in his
study of B-branes on Landau-Ginzburg models in the framework of the Homological
Mirror Symmetry Conjecture. The singularity category measures the homological
singularity of an algebra in the sense that an algebra A has finite global dimension
if and only if its singularity category $\mathbf{D}_{sg}(A)$ vanishes.\\
A celebrated theorem of Buchweitz (see \cite{Buchweitz}) shows that if $R$ is an Iwanaga-Gorenstein ring, then the stable
category of Cohen-Macaulay R-modules is triangle equivalent to the singularity category of $R$.\\
Let $\mathcal{C}$ be an additive category. We denote by $\mathrm{Mod}(\mathcal{C})$ the category of left $\mathcal{C}$-modules. The notion of singular equivalences for rings is further extended to
additive categories $\mathcal{C}$ by using $\mathrm{Mod}(\mathcal{C})$ as follows: We take the Verdier quotient
$$\mathbf{D}_{sg}(\mathrm{Mod}(\mathcal{C}))=\mathbf{D}^{b}(\mathrm{Mod}(\mathcal{C}))/\mathrm{Perf}(\mathrm{Mod}(\mathcal{C}))$$ and we call this the singularity category of $\mathcal{C}$. We say that two additive categories $\mathcal{C}$ and $\mathcal{C'}$ are singularly equivalent if there exists a triangle equivalence
$$\mathbf{D}_{sg}(\mathrm{Mod}(\mathcal{C}))\simeq \mathbf{D}_{sg}(\mathrm{Mod}(\mathcal{C'})).$$
Then one natural question is: when two additive categories are singularly equivalent? In general this is a difficult question.\\ The main purpose of this paper is to explore this question for certain additive categories  that are quotient by strongly idempotent ideals $\mathcal{I}$. We show that if $\mathcal{I}$ is a strongly idempotent ideal which has a finite projective dimension in $\mathrm{Mod}(\mathcal{C}^{e})$,  then there exists a singular equivalence between $\mathcal{C}$ and $\mathcal{C}/I$ (see Theorem \ref{teoremaprincipal}). This result is a generalization of the result obtained in \cite{Xiao} by Xiao-Wu Chen. It is well known that strongly idempotent ideals appears in the study of triangular matrix categories (see \cite{EdgarValente}). In particular, the one point-extension category is a triangular matrix category. In the final part of this paper we show that if one consider the one-point extension category $\Lambda:=\left[ \begin{smallmatrix}
\mathcal{C}_{K} & 0 \\ 
\underline{M} & \mathcal{U}
\end{smallmatrix}\right],$ then there exists an equivalence of triangulated categories $\mathbf{D}_{sg}(\mathrm{Mod}(\Lambda))\simeq \mathbf{D}_{sg}(\mathrm{Mod}(\mathcal{U}))$ (see Corollary \ref{singularonepoint}). We give an explicit example in the context of representation of infinite quivers.

\section{Preliminaries}
Throughout this paper we will consider small $K$-categories  $\mathcal{C}$ over a field $K$, which means that the class of objects of $\mathcal{C}$ forms a set, the morphisms set $\mathrm{Hom}_{\mathcal{C}}(X,Y)$ is a $K$-vector space and the composition of morphisms is $K$-bilinear.  For conciseness, we will sometimes write $\mathcal{C}(X,Y)$ instead of $\mathrm{Hom}_{\mathcal{C}}(X,Y)$. Furthermore, we refer to \cite{Mitchelring} for basic properties of $K$-categories.\\
Let $\mathcal{A}$ and $\mathcal{B}$ be $K$-categories. A covariant $K$-functor is a funtor $F:\mathcal{A}\rightarrow \mathcal{B}$ such that $F:\mathcal{A}(X,Y)\rightarrow \mathcal{B}(F(X),F(Y))$ is a $K$-linear transformation.
For $K$-categories $\mathcal{A}$ and  $\mathcal{B}$, we consider the category of all the covariant $K$-functors, which we denote by $\mathrm{Fun}_{K}(\mathcal{A},\mathcal{B})$. Given an arbitrary small additive category $\mathcal{C}$, the category of all additive covariant functors  $\mathrm{Fun}_{\mathbb{Z}}(\mathcal{C},\mathbf{Ab})$ is denoted by $\mathrm{Mod}(\mathcal{C})$ and is called the category of left $\mathcal{C}$-modules. When
 $\mathcal{C}$ is a $K$-category, there is an isomorphism of categories $\mathrm{Fun}_{\mathbb{Z}}(\mathcal{C},\mathbf{Ab})\simeq \mathrm{Fun}_{K}(\mathcal{C},\mathrm{Mod}(K))$ where $\mathrm{Mod}(K)$ denotes the category of $K$-vector spaces. Thus, we can identify $\mathrm{Mod}(\mathcal{C})$ with $\mathrm{Fun}_{K}(\mathcal{C},\mathrm{Mod}(K))$.  If $\mathcal{C}$ is a $K$-category, we always consider its opposite  category $\mathcal{C}^{op}$, which is also a $K$-category; and we construct the category of right $\mathcal{C}$-modules $\mathrm{Mod}(\mathcal{C}^{op}):=\mathrm{Fun}_{K}(\mathcal{C}^{op},\mathrm{Mod}(K))$. It is well-known that $\mathrm{Mod}(\mathcal{C})$ is an abelian category with enough projectives and injectives; see for example,\cite[Proposition 2.3]{MitBook} on page 99 and also page 102 in \cite{MitBook}.\\
 
If $\mathcal{C}$ and $\mathcal{D}$ are $K$-categories, B. Mitchell defined in \cite{Mitchelring} the $K$-category tensor product  $\mathcal{C}\otimes_{K}\mathcal{D}$ with objects that are those of $\mathcal{C}\times \mathcal{D}$, and the set of morphisms from $(C,D)$ to $(C',D')$ is the tensor product of $K$-vector spaces $\mathcal{C}(C,C')\otimes_{K}\mathcal{D}(D,D')$. The $K$-bilinear composition in $\mathcal{C}\otimes_{K} \mathcal{D}$ is given as follows: $(f_{2}\otimes g_{2})\circ (f_{1}\otimes g_{1}):=(f_{2}\circ f_{1})\otimes(g_{2}\circ g_{1})$ 
for all $f_{1}\otimes g_1\in \mathcal{C}(C,C')\otimes \mathcal{D}(D,D')$ and  $f_{2}\otimes g_2\in\mathcal{C}(C',C'')\otimes_{K} \mathcal{D}(D',D'')$.\\

Now we recall an important construction given in  \cite{Mitchelring} on p. 26 that will be used throughout this paper. Let $\mathcal{C}$ and $\mathcal{A}$ be  $K$-categories where $\mathcal{A}$ is cocomplete. The evaluation $K$-functor $E:\mathrm{Fun}_{K}(\mathcal{C}^{op},\mathcal{A})\otimes_{K}\mathcal{C}\longrightarrow \mathcal{A}$ can be extended to a $K$-functor
\begin{equation}\label{tensorMitchel}
-\otimes_{\mathcal{C}}-:\mathrm{Fun}_{K}(\mathcal{C}^{op},\mathcal{A})\otimes_{K}\mathrm{Mod}(\mathcal{C})\longrightarrow \mathcal{A}.
\end{equation}
By definition, we have an isomorphism $F\otimes_{\mathcal{C}}\mathcal{C}(X,-)\simeq F(X)$ for all $X\in \mathcal{C}$,
which is natural in $F$ and $X$.\\

\subsection{Derived categories}
Let $\mathcal{A}$ be an additive category, and let $K(\mathcal{A})$ be the homotopy category of $\mathcal{A}$. The subcategories
$K^{+}(\mathcal{A})$, $K^{-}(\mathcal{A})$ and $K^{b}(\mathcal{A})$ of $K(\mathcal{A})$ are
generated by the bounded below complexes, the bounded above complexes, and the bounded complexes, respectively.
For an abelian category $\mathcal{A}$, the derived category $\mathbf{D}(\mathcal{A})$ (resp. $\mathbf{D}^{+}(\mathcal{A})$, $\mathbf{D}^{-}(\mathcal{A})$ and $\mathbf{D}^{b}(\mathcal{A})$) is the quotient of $K(\mathcal{A})$ (resp. $K^{+}(\mathcal{A})$, $K^{-}(\mathcal{A})$ and $K^{b}(\mathcal{A})$) by the multiplicative set of quasi-isomorphisms.
Therefore $K^{\ast}(\mathcal{A})$ and $\mathbf{D}^{\ast}(\mathcal{A})$ are triangulated categories, where 
$$\ast=\text{nothing},+,-,\,\text{or}\,b,$$
see (\cite{Hartshorne}, \cite{Verdier}).\\
In general, we denote $K^{\ast}(\mathcal{A})$ as a localizing subcategory of $K(\mathcal{A})$, meaning that $K^{\ast}(\mathcal{A})$  is a full triangulated subcategory of $K(\mathcal{A})$, and the functor $\mathbf{D}^{\ast}(\mathcal{A})\longrightarrow  \mathbf{D}(\mathcal{A})$ is fully faithfull, where $\mathbf{D}^{\ast}(\mathcal{A})$ is the quotient of $K^{\ast}(\mathcal{A})$ by a multiplicative set of quasi-isomorphisms ( \cite[I, Sect. 5]{Hartshorne}, \cite[II, Sect. 1, No. 1]{Verdier}).
For further details on the triangulated structure of  $\mathbf{D}^{\ast}(\mathcal{A})$ see, for example, \cite{Gelfand}.
Let $\mathcal{T}$ be a triangulated category with equivalence  $\Sigma$. A non-empty full subcategory $\mathcal{S}$ of $\mathcal{T}$ is a triangulated subcategory if the following conditions hold.
\begin{enumerate}
\item [(a)] $\Sigma^{n}(X)\in \mathcal{S}$ for all $X\in \mathcal{S}$ and for all $n\in \mathbb{Z}$,

\item [(b)] Let $\xymatrix{X\ar[r] & Y\ar[r] & Z\ar[r] & \Sigma(X)}$ be a triangle in $\mathcal{T}$.  If two objects of
$\{X, Y, Z\}$ belong to $\mathcal{S}$, then also the third.
\end{enumerate}
A triangulated subcategory $\mathcal{S}$ of $\mathcal{T}$ is $\textbf{thick}$ if, for any morphisms  $\xymatrix{X\ar[r]^{\pi} & Y\ar[r]^{i} & X}$  in $\mathcal{T}$ where  $\pi\circ i=1_{Y}$ and $X\in \mathcal{S}$, it follows that $Y\in \mathcal{S}$

\section{Homological epimorphisms in functor categories}

A  $\textbf{two}$ $\textbf{sided}$ $\textbf{ideal}$  $\mathcal{I}(-,?)$ of $\mathcal{C}$ is a $K$-subfunctor of the two variable functor $\mathcal{C}(-,?):\mathcal{C}^{op}\otimes_{K}\mathcal{C}\rightarrow\mathrm{Mod}(K)$, such that the following conditions hold: (a) if $f\in \mathcal{I}(X,Y)$ and $g\in\mathcal C(Y,Z)$, then  $gf\in \mathcal{I}(X,Z)$; and (b)
if $f\in \mathcal{I}(X,Y)$ and $h\in\mathcal{C}(U,X)$, then  $fh\in \mathcal{I}(U,Z)$. If $\mathcal{I}$ is a two-sided ideal,  we can form the $\textbf{quotient category}$  $\mathcal{C}/\mathcal{I}$, whose objects are those of $\mathcal{C}$ and where $(\mathcal{C}/\mathcal{I})(X,Y):=\mathcal{C}(X,Y)/\mathcal{I}(X,Y)$, with composition induced by that of $\mathcal{C}$ (see \cite{Mitchelring}). There is a canonical projection functor $\pi:\mathcal{C}\rightarrow \mathcal{C}/\mathcal{I}$ such that $\pi(X)=X$ for all $X\in \mathcal{C}$ and  $\pi(f)=f+\mathcal{I}(X,Y):=\bar{f}$ for all $f\in \mathcal{C}(X,Y)$. We also recall that there exists a canonical isomorphism of categories $(\mathcal{C}/\mathcal{I})^{op}\simeq \mathcal{C}^{op}/\mathcal{I}^{op}$.\\
By taking $\mathcal{A}=\mathrm{Mod}(K)$ in equation \ref{tensorMitchel},  we have a functor 
$$-\otimes_{\mathcal{C}}-:\mathrm{Mod}(\mathcal{C}^{op})\times \mathrm{Mod}(\mathcal{C})\longrightarrow \mathrm{Mod}(K).$$
For properties of this tensor product we refer the reader to \cite{AuslanderRep1}.
Therefore,  for $N\in \mathrm{Mod}(\mathcal{C}^{op})$  we consider the functor
$N\otimes_{\mathcal{C}}-:\mathrm{Mod}(\mathcal{C})\longrightarrow \mathrm{Mod}(K)$. We denote by $\mathrm{Tor}_{i}^{\mathcal{C}}(N,-):\mathrm{Mod}(\mathcal{C})\longrightarrow \mathrm{Mod}(K)$ the $i$-th left derived functor of $N\otimes_{\mathcal{C}} -$. For $M\in \mathrm{Mod}(\mathcal{C})$ we now denote by
$\mathrm{Ext}^{i}_{\mathrm{Mod}(\mathcal{C})}(M,-):\mathrm{Mod}(\mathcal{C})\longrightarrow \mathrm{Mod}(K)$ the $i$-th right derived functor of $\mathrm{Hom}_{\mathrm{Mod}(\mathcal{C})}(M,-):\mathrm{Mod}(\mathcal{C})\longrightarrow \mathrm{Mod}(K)$.\\

We recall the construction of the following functors given in \cite[Definition 3.9]{RSS} and \cite[Definition 3.10]{RSS}.
The functor $ \frac{\mathcal{C}}{\mathcal{I}}\otimes_{\mathcal{C}}-:\mathrm{Mod}(\mathcal{C})\longrightarrow\mathrm{Mod}(\mathcal{C}/\mathcal{I}) $  is given as follows: for $M\in \mathrm{Mod}(\mathcal{C})$, we set 
$\left( \frac{\mathcal{C}}{\mathcal{I}}\otimes _{\mathcal{C}}M\right)(C):= \frac{\mathcal{C}(-,C)}{\mathcal{I}(-,C)}\otimes_{\mathcal{C}}M$  for all $C\in \mathcal{C}/\mathcal{I}$. We also define the functor $\mathcal{C}(\frac{\mathcal{C}}{\mathcal{I}},-):\mathrm{Mod}\left(\mathcal{C}\right) \longrightarrow \mathrm{Mod}\left(\mathcal{C}/\mathcal{I}\right) $ as follows: for $ M\in \mathrm{Mod}(\mathcal{C})$, we set 
$\mathcal{C}(\frac{\mathcal{C}}{\mathcal{I}},M)(C):=\mathcal{C}\left(\frac{\mathcal{C}(C,-)}{\mathcal{I}(C,-)},M\right)$ for all $C\in \mathcal{C}/\mathcal{I}$.

\begin{definition}$\textnormal{\cite[Definition 3.15]{RSS}}$
We denote by $\mathbb{EXT}^{i}_{\mathcal{C}}(\mathcal{C}/\mathcal{I},-):\mathrm{Mod}(\mathcal{C})\rightarrow \mathrm{Mod}(\mathcal{C}/\mathcal{I})$ the $i$-th right derived functor of $\mathcal{C}(\frac{\mathcal{C}}{\mathcal{I}},-)$ and by $\mathbb{TOR}_{i}^{\mathcal{C}}(\mathcal{C}/\mathcal{I},-):\mathrm{Mod}(\mathcal{C})\rightarrow \mathrm{Mod}(\mathcal{C}/\mathcal{I})$ the $i$-th left derived functor of $ \frac{\mathcal{C}}{\mathcal{I}}\otimes_{\mathcal{C}}$.
\end{definition}

We have the following description of the above functors
\begin{remark}\label{descEXT}
Consider the functors $\mathbb{EXT}^{i}_{\mathcal{C}}(\mathcal{C}/\mathcal{I},-):\mathrm{Mod}(\mathcal{C})\longrightarrow \mathrm{Mod}(\mathcal{C}/\mathcal{I})$ and $\mathbb{TOR}_{i}^{\mathcal{C}}(\mathcal{C}/\mathcal{I},-):\mathrm{Mod}(\mathcal{C})\longrightarrow \mathrm{Mod}(\mathcal{C}/\mathcal{I})$. The following  holds.
\begin{enumerate}
\item [(a)] For $M\in\mathrm{Mod}(\mathcal{C})$ we have that $\mathbb{EXT}^{i}_{\mathcal{C}}(\mathcal{C}/\mathcal{I},M)(C)=\mathrm{Ext}^{i}_{\mathrm{Mod}(\mathcal{C})}\left(\frac{\mathrm{Hom}_{\mathcal{C}}(C,-)}{\mathcal{I}(C,-)},M\right)$ for every $C\in \mathcal{C}/\mathcal{I}$.

\item [(b)] For $M\in\mathrm{Mod}(\mathcal{C})$ we have that $\mathbb{TOR}_{i}^{\mathcal{C}}(\mathcal{C}/\mathcal{I},M)(C)=\mathrm{Tor}_{i}^{\mathcal{C}}\left(\frac{\mathrm{Hom}_{\mathcal{C}}(-,C)}{\mathcal{I}(-,C)},M\right)$ for every $C\in \mathcal{C}/\mathcal{I}$. 

\end{enumerate}
\end{remark}
Let us consider $\pi_{1}:\mathcal{C}\longrightarrow \mathcal{C}/\mathcal{I}$ and $\pi_{2}:\mathcal{C}^{op}\longrightarrow \mathcal{C}^{op}/\mathcal{I}^{op}$ the canonical projections. From Section 5 in \cite{RSS}, we obtain the following definition, which is a generalization of a notion given for artin algebras by Auslander-Platzeck-Todorov in  \cite{APG}.

\begin{definition}$\textnormal{\cite[Definition 5.1]{RSS}}$ \label{kidemcat}
Let $\mathcal{C}$ be a $K$-category  and let $\mathcal{I}$ be an ideal in $\mathcal{C}$.
We say that $\mathcal{I}$ is $\textbf{strongly idempotent}$ if 
$$\varphi^{i}_{F,(\pi_{1})_{\ast}(F')}:\mathrm{Ext}^{i}_{\mathrm{Mod}(\mathcal{C}/\mathcal{I})}(F,F')\longrightarrow \mathrm{Ext}^{i}_{\mathrm{Mod}(\mathcal{C})}((\pi_{1})_{\ast}(F),(\pi_{1})_{\ast}(F'))$$ is an isomorphism for all $F,F'\in \mathrm{Mod}(\mathcal{C}/\mathcal{I})$ and for all $0\leq i < \infty$.
\end{definition}

Now,
from section 5 in \cite{RSS}, for $F\in \mathrm{Mod}((\mathcal{C}/\mathcal{I})^{op})$  and  $F'\in \mathrm{Mod}(\mathcal{C}/\mathcal{I})$ we have the morphism $\psi_{F,(\pi_{1})_{\ast}(F')}^{i}:\mathrm{Tor}^{\mathcal{C}}_{i}(F\circ \pi_{2},F'\circ \pi_{1})\longrightarrow \mathrm{Tor}^{\mathcal{C}/\mathcal{I}}_{i}(F, F')$. 
We obtain the following result that is a kind of generalization of Theorem 4.4 of Geigle and Lenzing in \cite{GeigleLen}.

\begin{proposition}\label{caractidem2}$\textnormal{\cite[Proposition 3.4]{EdgarValente}}$
Let $\mathcal{C}$ be a  $K$-category  and $\mathcal{I}$ an ideal.  The following are equivalent.
\begin{enumerate}
\item [(a)] $\mathcal{I}$ is strongly idempotent 

\item [(b)] $\mathbb{EXT}^{i}_{\mathcal{C}}(\mathcal{C}/\mathcal{I},F'\circ \pi_{1})=0$ for $1\leq i<\infty$ and for $F'\in \mathrm{Mod}(\mathcal{C}/\mathcal{I})$.

\item [(c)] $\mathbb{EXT}^{i}_{\mathcal{C}}(\mathcal{C}/\mathcal{I},J\circ \pi_{1})=0$ for $1\leq i<\infty$ and for each $J\in \mathrm{Mod}(\mathcal{C}/\mathcal{I})$ which is injective.

\item [(d)] $\psi_{F,(\pi_{1})_{\ast}(F')}^{i}:\mathrm{Tor}^{\mathcal{C}}_{i}(F\circ \pi_{2},F'\circ \pi_{1})\longrightarrow \mathrm{Tor}^{\mathcal{C}/\mathcal{I}}_{i}(F, F')$ is an isomorphism for all $0\leq i<\infty$ and $F\in \mathrm{Mod}((\mathcal{C}/\mathcal{I})^{op})$ as well as  $F'\in \mathrm{Mod}(\mathcal{C}/\mathcal{I})$.

\item [(e)] $\mathbb{TOR}_{i}^{\mathcal{C}}(\mathcal{C}/\mathcal{I},F'\circ\pi_{1})=0$ for $1\leq i<\infty$  and for all $F'\in \mathrm{Mod}(\mathcal{C}/\mathcal{I})$.

\item [(f)] $\mathbb{TOR}_{i}^{\mathcal{C}}(\mathcal{C}/\mathcal{I},P\circ\pi_{1})=0$ for $1\leq i<\infty$ and for all $P\in \mathrm{Mod}(\mathcal{C}/\mathcal{I})$ which is projective.

\item [(g)] The canonical functor $\pi_{\ast}:\mathbf{D}^{b}(\mathrm{Mod}(\mathcal{C}/\mathcal{I}))\longrightarrow \mathbf{D}^{b}(\mathrm{Mod}(\mathcal{C}))$
is full and faithful.

\end{enumerate}
\end{proposition}
\begin{proof}
The proof given in \cite[Corollary 5.10]{RSS} can be adapted to this setting.
\end{proof}

The following is a generalization of \cite[Definition 4.5]{GeigleLen}.

\begin{definition}$\textnormal{\cite[Definition 3.5]{EdgarValente}}$
Let $\mathcal{I}$ be an ideal of $\mathcal{C}$. It is said that $\pi_{1}:\mathcal{C}\longrightarrow\mathcal{C}/\mathcal{I}$ is an $\textbf{homological epimorphism}$ if $\mathcal{I}$ is strongly idempotent.
\end{definition}

\begin{proposition}\label{proyepihomo}$\textnormal{\cite[Proposition 4.4]{EdgarValente}}$
Let $\mathcal{I}$ be an idempotent ideal of $\mathcal{C}$ such that $\mathcal{I}(C,-)$ is projective in $\mathrm{Mod}(\mathcal{C})$ for all $C\in \mathcal{C}$.  Then $\mathcal{I}$ is strongly idempotent.
\end{proposition}

The following definition can be found on page 56 in \cite{Mitchelring}.
\begin{definition}
Let $\mathcal{C}$ be a $K$-category.  The $\textbf{enveloping category}$ of $\mathcal{C}$, denoted by $\mathcal{C}^{e}$, is defined as $\mathcal{C}^{e}:=\mathcal{C}^{op}\otimes_{K}\mathcal{C}.$
\end{definition}

Consider an ideal $\mathcal{I}$ of $\mathcal{C}$ and  $\pi:\mathcal{C}\longrightarrow \mathcal{C}/\mathcal{I}=\mathcal{B}$ the canonical epimorphism. Consider
$H:=\mathcal{B}(-,-)\circ (\pi^{op}\otimes \pi)$. Thus, we obtain a morphism in $\mathrm{Mod}(\mathcal{C}^{e})$:
$$\Gamma(\pi):\mathcal{C}(-,-)\longrightarrow \mathcal{B}(-,-)\circ (\pi^{op}\otimes \pi)$$
such that for an object $(C,C')\in \mathcal{C}^{e}$, we have that $[\Gamma(\pi)]_{(C,C')}:\mathcal{C}(C,C')\longrightarrow \mathcal{B}(\pi(C),\pi(C'))$ is defined as $[\Gamma(\pi)]_{(C,C')}(f):=\pi(f)$ for all $f\in \mathcal{C}(C,C')$. Thus, we obtain the following exact sequence in $\mathrm{Mod}(\mathcal{C}^{e})$:

\begin{equation}\label{succanonica}
\xymatrix{0\ar[r] & \mathcal{I}\ar[r] & \mathcal{C}\ar[r]^{\Gamma(\pi)} & H \ar[r]  & 0.}
\end{equation}

\section{Singularity category}

The singularity category $\mathbf{D}_{sg}(A)$ of an algebra A over a field $K$, introduced by
R.O. Buchweitz in \cite{Buchweitz}, is defined as the Verdier quotient 
$$\mathbf{D}_{sg}(A)=\mathbf{D}^{b}(\mathrm{mod}(A))/\mathrm{perf}(A)$$ of the bounded derived category $\mathbf{D}^{b}(\mathrm{mod}(A))$ by the category of perfect complexes.
In recent years, D. Orlov (\cite{Orlov1} ) rediscovered the notion of singularity categories in his
study of B-branes on Landau-Ginzburg models in the framework of the Homological
Mirror Symmetry Conjecture. The singularity category measures the homological
singularity of an algebra in the sense that an algebra A has finite global dimension
if and only if its singularity category $\mathbf{D}_{sg}(A)$ vanishes.\\
Let $\mathcal{A}$  be an abelian category with enough projective objects. We denote by $\mathrm{Perf}(\mathcal{A})$ the full subcategory of $\mathbf{D}^{b}(\mathcal{A})$ consisting of complexes isomorphic  in $\mathbf{D}^{b}(\mathcal{A})$  to a bounded complex $P^{\bullet}$ of projective objects of $\mathcal{A}$. It is easy to see that  $\mathrm{Perf}(\mathcal{A})$ is a thick triangulated subcategory of $\mathbf{D}^{b}(\mathcal{A})$.

\begin{definition}$\textnormal{\cite[Definition  in p. 3768]{Zhao}}$
Let $\mathcal{A}$  be an abelian category with enough projective objects.
The singularity category of $\mathcal{A}$ is defined to be the
following Verdier quotient triangulated category
$$\mathbf{D}_{sg}(\mathcal{A})=\mathbf{D}^{b}(\mathcal{A})/\mathrm{Perf}(\mathcal{A}).$$
\end{definition}
For the construction of the Verdier's quotient see for example \cite{Thorsten} or \cite{Neeman}.

%In particular for an additive category $\mathcal{C}$ we define
%$$\mathbf{D}_{sg}(\mathcal{C})=\mathbf{D}^{b}(\mathrm{Mod}(\mathcal{C}))/\mathrm{Perf}(\mathrm{Mod}(\mathcal{C}))$$
%For the singularity category of an exact category see for example \cite{estrada}.
%(ver pag. 3 en la version del arxiv). Also see \cite[section 6.1]{Panagiotis} for the definition of the singularity category of an abelian category. See \cite[Corollary 3.9]{Apostolos} for some properties.

\begin{remark} Let $\Lambda$ be a ring. It is importan to consider $\mathbf{D}_{sg}(\mathcal{A})$ where $\mathcal{A}=\mathrm{Mod}(\Lambda)$ instead of just $\mathcal{A}=\mathrm{mod}(\Lambda)$  (the category of finitely generated $\Lambda$-modules). Because $\mathbf{D}_{sg}(\mathrm{Mod}(\Lambda))$  is the category that measures de singularity of $\Lambda$ in the sense that $\mathbf{D}_{sg}(\mathrm{Mod}(\Lambda))=0$ if and only if $\mathrm{gl.dim}(\Lambda)<\infty$, for any ring $\Lambda$ (see Remark 6.9 in 
\cite{Panagiotis}.) 
\end{remark}

\section{Main Theorem}

Let $I$ be an ideal of a $K$-category $\mathcal{C}$ and consider the functor $\pi:\mathcal{C}\longrightarrow \mathcal{C}/I$. Recall that  $\mathcal{C}/I\in \mathrm{Mod}((\mathcal{C}/I)^{e})$, that is, $(\mathcal{C}/I)(-,-):(\mathcal{C}/I)^{op}\otimes_K \mathcal{C}/I\longrightarrow \mathrm{Mod}(K)$. We also get the  induced functor
$$\pi_{\ast}:\mathrm{Mod}(\mathcal{C}/I)\longrightarrow \mathrm{Mod}(\mathcal{C}).$$

\begin{definition}\label{unbifuntor}
Let $\mathcal{C}$ and $\mathcal{D}$ be $K$-categories. There is a  bifunctor 
$$F:=-\boxtimes_{\mathcal{C}}-: \mathrm{Mod}(\mathcal{C}^{op}\otimes_K \mathcal{D})\times \mathrm{Mod}(\mathcal{C})\longrightarrow \mathrm{Mod}(\mathcal{D})$$
where for $B\in \mathrm{Mod}(\mathcal{C}^{op}\otimes_K \mathcal{D})$, $X\in \mathrm{Mod}(\mathcal{C})$ and $D\in \mathcal{D}$ we set
$$(B\boxtimes_{\mathcal{C}}X)(D):=B(-,D)\otimes_{\mathcal{C}}X.$$
\end{definition}

We have the following proposition.

\begin{proposition}\label{traingleconmuta}
Consider  $H:=(\mathcal{C}/I)\circ (\pi^{op}\otimes \pi)\in \mathrm{Mod}(\mathcal{C}^{e})$  and $H_{1}:=(\mathcal{C}/I)\circ (\pi^{op}\otimes 1)\in \mathrm{Mod}(\mathcal{C}^{op}\otimes_K (\mathcal{C}/I))$. Then the following diagram commutes
$$\xymatrix{\mathrm{Mod}(\mathcal{C})\ar[rr]^{H\boxtimes_{\mathcal{C}}-}\ar[dr]_{H_{1}\boxtimes_{\mathcal{C}}-} & & \mathrm{Mod}(\mathcal{C})\\
 & \mathrm{Mod}(\mathcal{C}/I)\ar[ur]_{\pi_{\ast}}}$$
\end{proposition}
\begin{proof}
It is straightforward.
\end{proof}

\begin{remark}\label{2bifuntores}
We note that $H_{1}\boxtimes_{\mathcal{C}}-$ is the same as the funtor 
$$\mathcal{C}/I\otimes_{\mathcal{C}}-: \mathrm{Mod}(\mathcal{C})\longrightarrow \mathrm{Mod}(\mathcal{C}/I)$$ defined in p. 793 in \cite{LGOS2}. The functor $ \frac{\mathcal{C}}{I}\otimes_{\mathcal{C}}:\mathrm{Mod}(\mathcal{C})\longrightarrow\mathrm{Mod}(\mathcal{C}/I) $ is defined as follows: 
$\left( \frac{\mathcal{C}}{I}\otimes _{\mathcal{C}}M\right)(C):= \frac{\mathcal{C}(-,C)}{I(-,C)}\otimes_{\mathcal{C}}M$  for all $M\in \mathrm{Mod}(\mathcal{C})$
and  $\left( \frac{\mathcal{C}}{I}\otimes _{\mathcal{C}}M\right)(\overline{f}):=\frac{\mathcal{C}}{I}(-,f)\otimes_{\mathcal{C}}M$ for all $\overline{f}=f+I(C,C')\in \frac{\mathcal{C}(C,C')}{I(C,C')}$.
\end{remark}

Consider the bifuntor given in Definition \ref{unbifuntor}:
$$F:=-\boxtimes_{\mathcal{C}}-:\mathrm{Mod}(\mathcal{C}^{op}\otimes_K \mathcal{D})\times \mathrm{Mod}(\mathcal{C})\longrightarrow \mathrm{Mod}(\mathcal{D}).$$

Now, by following the construction in p. 57 in \cite{Kashiwara} but for the case of right exact bifuntors, we have the induced bifunctor
$$\mathbb{F}:=K^{-}F:\mathbf{K}^{-}\Big(\mathrm{Mod}(\mathcal{C}^{op}\otimes_K \mathcal{D})\Big)\times \mathbf{K}^{-}\Big(\mathrm{Mod}(\mathcal{C})\Big)\longrightarrow \mathbf{K}^{-}(\mathrm{Mod}(\mathcal{D})).$$

\begin{proposition}\label{funtorderivadotensor}
Consider the bifuntor given in Definition \ref{unbifuntor}:
$$F:=-\boxtimes_{\mathcal{C}}-:\mathrm{Mod}(\mathcal{C}^{op}\otimes_K \mathcal{D})\times \mathrm{Mod}(\mathcal{C})\longrightarrow \mathrm{Mod}(\mathcal{D}).$$
Then, there exists the left derived bifunctor of $\mathbb{F}$:
$$L^{-}\mathbb{F}=-\boxtimes_{\mathcal{C}}^{L}-:\mathbf{D}^{-}\Big(\mathrm{Mod}(\mathcal{C}^{op}\otimes_K \mathcal{D})\Big)\times \mathbf{D}^{-}\Big(\mathrm{Mod}(\mathcal{C})\Big)\longrightarrow \mathbf{D}^{-}\Big(\mathrm{Mod}(\mathcal{D})\Big).$$
Moreover,  the following statements hold.
\begin{enumerate}
\item [(a)]
For $X^{\bullet}\in \mathbf{K}^{-}\Big(\mathrm{Mod}(\mathcal{C}^{op}\otimes_K \mathcal{D})\Big)$ the functor
$$\mathbb{F}(X^{\bullet},-): \mathbf{K}^{-}\Big(\mathrm{Mod}(\mathcal{C})\Big)\longrightarrow \mathbf{K}^{-}(\mathrm{Mod}(\mathcal{D}))$$ has a left derived functor
$$L_{II}^{-}\mathbb{F}(X^{\bullet},-): \mathbf{D}^{-}\Big(\mathrm{Mod}(\mathcal{C})\Big)\longrightarrow \mathbf{D}^{-}(\mathrm{Mod}(\mathcal{D})).$$

\item [(b)] For $Y^{\bullet}\in \mathbf{K}^{-}\Big(\mathrm{Mod}(\mathcal{C})\Big)$    the functor
$$\mathbb{F}(-,Y^{\bullet}): \mathbf{K}^{-}\Big(\mathrm{Mod}(\mathcal{C}^{op}\otimes_K \mathcal{D})\Big)\longrightarrow \mathbf{K}^{-}(\mathrm{Mod}(\mathcal{D}))$$ has a left derived functor
$$L_{I}^{-}\mathbb{F}(-,Y^{\bullet}): \mathbf{D}^{-}\Big(\mathrm{Mod}(\mathcal{C}^{op}\otimes_K \mathcal{D})\Big)\longrightarrow \mathbf{D}^{-}(\mathrm{Mod}(\mathcal{D})).$$

\item [(c)] For  $X^{\bullet}\in \mathbf{D}^{-}\Big(\mathrm{Mod}(\mathcal{C}^{op}\otimes_K \mathcal{D})\Big)$ and   $Y^{\bullet}\in \mathbf{D}^{-}\Big(\mathrm{Mod}(\mathcal{C})\Big)$, there exist isomorphisms:
$$L^{-}\mathbb{F}(X^{\bullet},Y^{\bullet})\simeq L_{II}^{-}\mathbb{F}(X^{\bullet},Y^{\bullet})\simeq L_{I}^{-}\mathbb{F} (X^{\bullet},Y^{\bullet}).$$
\end{enumerate}

\end{proposition}
\begin{proof}
Let $\mathcal{P}=\mathrm{Proj}(\mathrm{Mod}(\mathcal{C}^{op}\otimes_K \mathcal{D}))$ and $\mathcal{P}'=\mathrm{Proj}(\mathrm{Mod}(\mathcal{C}))$  the category of projective modules in $\mathrm{Mod}(\mathcal{C}^{op}\otimes_K \mathcal{D})$ and $\mathrm{Mod}(\mathcal{C})$ respectively. Now, we have that the pair $(\mathcal{P},\mathcal{P}')$ is $\mathbb{F}$-projective (in the sense of the dual of Definition \cite[Definition 1.10.6]{Kashiwara} or \cite[Definition 13.4.2]{Kashiwara2}).
By dual of  \cite[Proposition 1.10.7]{Kashiwara}, we have that $\Big(\mathbf{K}^{-}\big(\mathcal{P}\big),\mathbf{K}^{-}\big(\mathcal{P}'\big)\Big)$ satisfies the duals of conditions 1.10.1 and 1.10.2 in p. 57 of \cite{Kashiwara}. Hence by dual of \cite[Proposition 1.10.4]{Kashiwara}, we have that there exists the left derived functor $L\mathbb{F}$.\\
We also have that  the subcategories $\mathbf{K}^{-}\big(\mathcal{P}\big)$ and $\mathbf{K}^{-}\big(\mathcal{P}'\big)$  satisfies the dual of the conditions 1.10.3 and 1.10.4 in p. 58 of \cite{Kashiwara}, for the functors $\mathbb{F}(-,Y^{\bullet})$  and  $\mathbb{F}(X^{\bullet},-)$ respectively, for every $X^{\bullet}\in \mathbf{K}^{-}\Big(\mathrm{Mod}(\mathcal{C}^{op}\otimes_K \mathcal{D})\Big)$ and $Y^{\bullet}\in \mathbf{K}^{-}\Big(\mathrm{Mod}(\mathcal{C})\Big)$. Therefore, by the duals of \cite[Corollary 1.10.5]{Kashiwara} and   \cite[Remark 1.10.10]{Kashiwara}, we have the result.\\
%Se also \cite[Corollary 13.4.5 ]{Kashiwara2} in p. 338-339.
\end{proof}

\begin{corollary}\label{compodosfuntor}
Consider $H=(\mathcal{C}/I)\circ (\pi^{op}\otimes \pi)\in \mathbf{K}^{-}\Big(\mathrm{Mod}(\mathcal{C}^{op}\otimes_K \mathcal{C})\Big)$ and the left derived functor $$L_{II}^{-}\mathbb{F}(H,-): \mathbf{D}^{-}\Big(\mathrm{Mod}(\mathcal{C})\Big)\longrightarrow \mathbf{D}^{-}\Big(\mathrm{Mod}(\mathcal{C})\Big)$$
and $H_{1}=(\mathcal{C}/I)\circ (\pi^{op}\otimes 1)\in \mathbf{K}^{-}\Big(\mathrm{Mod}((\mathcal{C}/I)^{op}\otimes_K \mathcal{C})\Big)$ and the left derived functor $$L_{II}^{-}\mathbb{F}(H_{1},-): \mathbf{D}^{-}\Big(\mathrm{Mod}(\mathcal{C})\Big)\longrightarrow \mathbf{D}^{-}\Big(\mathrm{Mod}(\mathcal{C}/I)\Big).$$
Then $$L_{II}^{-}\mathbb{F}(H,-)=L(\pi_{\ast})\circ L_{II}^{-}\mathbb{F}(H_{1},-).$$
\end{corollary}
\begin{proof}
 It follows from Proposition \ref{traingleconmuta} and the dual of Theorem 1 in p. 200 in \cite{Gelfand}.
\end{proof}

\subsection{Restricting functors to the bounded derived category and main theorem}

We now give the following definition, see for example first paragraph in p. 85 in \cite{Borel}
\begin{definition}
Let $F:\mathcal{A} \longrightarrow \mathcal{B}$ be a functor between abelian categories. We say that $F$ has \textbf{finite left cohomological dimension} if there exists an integer $n\geq 0$ such that 
$$L_{i}F(A)=H^{-i}(L^{-}F(A))=0$$
for all $A\in \mathcal{A}$ and for all $i>n$ ( we consider $A$ as a complex concentrated in zero degree), where
$L^{-}F:\mathbf{D}^{-}(\mathcal{A})\longrightarrow  \mathbf{D}^{-}(\mathcal{B})$ is the left derived functor of $F$. Dually, we say that $F$ has \textbf{finite right cohomological dimension} if there exists an integer $n\geq 0$ such that 
$$R_{i}F(A)=H^{i}(R^{+}F(A))=0$$
for all $A\in \mathcal{A}$ and for all $i>n$, where
$R^{+}F:\mathbf{D}^{+}(\mathcal{A})\longrightarrow  \mathbf{D}^{+}(\mathcal{B})$ is the right derived functor of $F$.
\end{definition}

The importance of finite left (co)homological dimension is that it allow us to restrict derived functors to the bounded derived categories.

\begin{lemma}\label{lema1Cline-Parshall}
Let $\mathcal{A}$ and $\mathcal{B}$ be abelian categories such that $\mathcal{A}$ has enough injectives and $\mathcal{B}$ has enough projectives. Let $F:\mathcal{A}\longrightarrow \mathcal{B}$ and $G:\mathcal{B}\longrightarrow \mathcal{A}$
additive functors such that
\begin{enumerate}
\item [(a)] $F$ is right adjoint to $G$,

\item [(b)] $F$ has finite right cohomological dimension and $G$ has finite left cohomological dimension.
\end{enumerate}
Then $R^{+}F:\mathbf{D}^{b}(\mathcal{A})\longrightarrow \mathbf{D}^{b}(\mathcal{B})$ is right adjoint to  $L^{-}G: \mathbf{D}^{b}(\mathcal{B})\longrightarrow \mathbf{D}^{b}(\mathcal{A})$.
\end{lemma}
\begin{proof}
See \cite[Lemma 1.1]{CPS1} in p. 399.
\end{proof}

%\begin{remark}
%In the definition of finite left cohomological dimension in \cite{Borel} it is used the notation $L_{-i}F(A)=H^{-i}(L^{-}F(A))$ with the sign minus in the first $i$, however we used  positive sign in the first $i$ in the above definition, in order to be consistent with the notation of the classical  left derived funtors $L_{i}F$.\\
%We also note that, it is used the canonical funtor $D:\mathcal{A}\longrightarrow D(\mathcal{A})$ (see 6.9 in p. 55 in \cite{Borel}).
%\end{remark}

In the following Lemma the hypothesis that $K$ is a field is crucial.

\begin{lemma}\label{proyectivoizqder}
Let $\mathcal{C}$ be a $K$-category and $P\in \mathrm{Mod}(\mathcal{C}^{e})$ be a projective $\mathcal{C}^{e}$-module. Then $P(-,C)\in \mathrm{Mod}(\mathcal{C}^{op})$ is a projective $\mathcal{C}^{op}$-module and $P(C,-)\in \mathrm{Mod}(\mathcal{C})$ is a projective $\mathcal{C}$-module.
\end{lemma}
\begin{proof}
Let us suppose that $P$ is a representable module, $P:=\mathcal{C}^{e}\big((A,B),-\big)$. Hence
$P(C,-)=\mathcal{C}^{e}\big((A,B),(C,-)\big)$. Then for $D\in \mathcal{C}$ we have that
$$P(C,D)=\mathcal{C}^{e}\big((A,B),(C,D)\big)=\mathcal{C}^{op}(A,C)\otimes_{K} \mathcal{C}(B,D)=\mathcal{C}(C,A)\otimes_{K} \mathcal{C}(B,D)$$
Now, we consider $\mathcal{C}(C,A)\otimes_{K} \mathcal{C}(B,-)\in \mathrm{Mod}(\mathcal{C})$. By  \cite[Corollary 11.7 in p. 55]{Mitchelring}, we have that  $\mathcal{C}(C,A)\otimes_{K} \mathcal{C}(B,-)$  is a projective $\mathcal{C}$-module.
Moreover, by the above calculation we have that
$$P(C,-)=\mathcal{C}(C,A)\otimes_{K}\mathcal{C}(B,-).$$
Hence $P(C,-)=\mathcal{C}^{e}\big((A,B),(C,-)\big)$  is a projective $\mathcal{C}$-module.\\
Now, let $P\in \mathrm{Mod}(\mathcal{C}^{e})$ be an arbitrary projective $\mathcal{C}^{e}$-module. Thus, there exist $Q$ such that
$$P\oplus Q=\bigoplus_{i\in I}\mathcal{C}^{e}\big((A_{i},B_{i}),-\big).$$
Then, for $C\in \mathcal{C}$ we have that
$$P(C,-)\oplus Q(C,-)=\bigoplus_{i\in I}\mathcal{C}^{e}\big((A_{i},B_{i}),(C,-)\big)$$
where each $\mathcal{C}^{e}\big((A_{i},B_{i}),(C,-)\big)$ is a projective $\mathcal{C}$-module. Hence $P(C,-)$ is  a projective $\mathcal{C}$-module.\\
Similarly, we can prove that  $P(-,C)\in \mathrm{Mod}(\mathcal{C}^{op})$ is a projective $\mathcal{C}^{op}$-module.
\end{proof}

\begin{lemma}\label{CIfinitedimension}
Let $I$ be an ideal of $\mathcal{C}$ such that the projective dimension of $I$ as $\mathcal{C}^{e}$-module is finite.   Then the projective dimensions of $\mathcal{C}(-,C)/I(-,C)\in \mathrm{Mod}(\mathcal{C}^{op})$  and  $\mathcal{C}(C,-)/I(C,-)\in \mathrm{Mod}(\mathcal{C})$ are finite.
\end{lemma}
\begin{proof}
We will do the case $\mathcal{C}(C,-)/I(C,-)$ since the other is similar.
Consider the exact sequence in $\mathrm{Mod}(\mathcal{C}^{e})$:
$$\xymatrix{0\ar[r] & I\ar[r] & \mathcal{C}\ar[r] & H\ar[r] & 0}$$
where  $H:=(\mathcal{C}/I)\circ (\pi^{op}\otimes \pi)\in  \mathrm{Mod}(\mathcal{C}^{e})$.\\
Suppose that $P^{\bullet}$ is a  finite projective resolution of $I$ in $\mathrm{Mod}(\mathcal{C}^{e})$:

$$\scalebox{0.9}{\xymatrix{P^{\bullet}: 0\ar[r] & P_{n}(-,-)\ar[r] & \cdots\ar[r] & P_{1}(-,-)\ar[r] & P_{0}(-,-)\ar[r] & I(-,-)\ar[r] & 0}}$$
Hence, by Lemma \ref{proyectivoizqder}, we have that  $P^{\bullet}(C,-)$ is a projective resolution of $I(C,-)$:
$$\scalebox{0.9}{\xymatrix{P^{\bullet}(C,-):0\ar[r] & P_{n}(C,-)\ar[r] & \cdots\ar[r] & P_{1}(C,-)\ar[r] & P_{0}(C,-)\ar[r] & I(C,-)\ar[r] & 0}}$$
It follows that $H(C,-)=\mathcal{C}(C,-)/I(C,-)$ has finite projective dimension in $\mathrm{Mod}(\mathcal{C})$. Indeed,  for each $C\in \mathcal{C}$ we have projective resolution of $\mathcal{C}(C,-)/I(C,-)$: 
$$\scalebox{0.9}{\xymatrix{0\ar[r] & P_{n}(C,-)\ar[r] & \cdots\ar[r] & P_{0}(C,-)\ar[r] &  \mathcal{C}(C,-)\ar[r] &\mathcal{C}(C,-)/I(C,-)\ar[r] & 0 }}$$
\end{proof}

\begin{lemma}\label{tensorfiniecomo}
Let $I$ be an ideal of $\mathcal{C}$ such that the projective dimension of $I$ as $\mathcal{C}^{e}$-module is finite. Then the  functor 
$$F:=\mathbb{F}(H_{1},-)=H_{1}\boxtimes_{\mathcal{C}}-:\mathrm{Mod}(\mathcal{C})\longrightarrow \mathrm{Mod}(\mathcal{C}/I)$$
has finite left cohomological dimension.
\end{lemma}
\begin{proof}
Using the notation  of Proposition \ref{funtorderivadotensor}, we have the left derived funtor $$L_{II}^{-}\mathbb{F}(H_{1},-):\mathbf{D}^{-}\Big(\mathrm{Mod}(\mathcal{C})\Big)\longrightarrow \mathbf{D}^{-}\Big(\mathrm{Mod}(\mathcal{C}/I)\Big)$$
By simplicity, let us denote
$$L^{-}F:=L_{II}^{-}\mathbb{F}(H_{1},-).$$
We  have that $H_{1}(-,C)=\mathcal{C}(-,C)/I(-,C)$ has finite projective dimension in $\mathrm{Mod}(\mathcal{C}^{op})$ (see Lemma \ref{CIfinitedimension}). Let $n:=\mathrm{pd}(H_{1}(-,C))$. We assert that 
$$L_{i}F(M):=H^{i}(L^{-}F(M)\Big)=0$$ for all $M\in \mathrm{Mod}(\mathcal{C})$ and for all $i>n$.\\
Indeed, we have that $H_{1}\boxtimes_{\mathcal{C}}-=\mathcal{C}/I\otimes_{\mathcal{C}}-$ (see Remark \ref{2bifuntores}).  Hence we have that
$$H^{-i}\circ L^{-}F\simeq L_{i}(H_{1}\boxtimes_{\mathcal{C}}-)=L_{i}(\mathcal{C}/I\otimes_{\mathcal{C}}-)$$
where $L_{i}(\mathcal{C}/I\otimes_{\mathcal{C}}-)$ is the i-th classical left derived funtor (see \cite[Corollary  10.5.7 and Remark 10.5.8]{Weibel}).
But according to \cite[Definition 3.15]{RSS}, we have that the $i$-th left  classical derived functor of $\mathcal{C}/I\otimes_{\mathcal{C}}-$  is $\mathbb{TOR}_{i}^{\mathcal{C}}(\mathcal{C}/I,-)$. Then

$$H^{-i}(L^{-}F(M)\Big)\simeq \mathbb{TOR}_{i}^{\mathcal{C}}(\mathcal{C}/I,M)\in \mathrm{Mod}(\mathcal{C}/I).$$
Hence, for $C\in\mathcal{C}/I$ we have that
$$\mathbb{TOR}_{i}^{\mathcal{C}}(\mathcal{C}/I,M)(C)=\mathrm{Tor}_{i}^{\mathcal{C}}\Big(\frac{\mathcal{C}(-,C)}{I(-,C)},M\Big)$$
(see Remark \ref{2bifuntores}).  Using the projective resolution $P^{\bullet}$ of $\mathcal{C}(-,C)/I(-,C)$ of length $n$, by definition we get that
$$\mathrm{Tor}_{i}^{\mathcal{C}}\Big(\frac{\mathcal{C}(-,C)}{I(-,C)},M\Big)=H^{-i}\Big(P_{\bullet}\otimes_{\mathcal{C}}M\Big).$$
Since $(P^{\bullet}\otimes_{\mathcal{C}}M)^{i}=P^{i}\otimes_{\mathcal{C}} M$ we have that $(P^{\bullet}\otimes_{\mathcal{C}}M)^{i}=0$ if $i>n$. Hence we have that $H^{i}\Big(P^{\bullet}\otimes_{\mathcal{C}}M\Big)=0$ for $i>n$.

Therefore, we conclude that
$$H^{-i}(L^{-}F(M)\Big)\simeq \mathbb{TOR}_{i}^{\mathcal{C}}(\mathcal{C}/I,M)=0$$
for all $M\in \mathrm{Mod}(\mathcal{C})$ and for all $i>n$. Proving that $F$ has finite left cohomological dimension.
\end{proof}

\begin{lemma}\label{triangululocan}
Consider a bounded  complex in $D^{b}(\mathcal{A})$:
$$\xymatrix{X^{\bullet}:\quad \ar[r] & 0\ar[r]  & X^{a}\ar[r] & X^{a+1}\ar[r] & \cdots \ar[r] & X^{b}\ar[r]  & 0\ar[r] & \cdots}$$
Consider the stupid truncation 

$$\xymatrix{\sigma^{>a}(X^{\bullet}):\quad \ar[r] & 0\ar[r] & 0\ar[r] & X^{a+1}\ar[r] & \cdots \ar[r] & X^{b}\ar[r]  & 0\ar[r] & \cdots}$$
Hence we have a triangle in the derived category $\mathbf{D}^{b}(\mathcal{A})$:
$$\xymatrix{\sigma^{>a}(X^{\bullet})\ar[r] &  X^{\bullet}\ar[r] & X^{a}[-a]\ar[r] & \sigma^{>a}(X^{\bullet})[1]}$$ where $X^{a}[-a]$ is the complex concentrated in degree $a$, such that in degree $a$ has the term $X^{a}$ and zero elsewhere.
\end{lemma}
\begin{proof}
We have the morphism of complexes

$$\xymatrix{\sigma^{>a}(X^{\bullet}):\quad \ar[r] & 0\ar[d]^{0}\ar[r] & 0\ar[d]^{0}\ar[r] & X^{a+1}\ar[r]\ar[d]^{1} & \cdots \ar[r] & X^{b}\ar[r]\ar[d]^{1}  & 0\ar[r]\ar[d] & \cdots\\
X^{\bullet}\quad \ar[r] & 0\ar[r]\ar[d] & X^{a}\ar[r]^{u}\ar[d]^{1} & X^{a+1}\ar[r]\ar[d]^{0} & \cdots \ar[r] & X^{b}\ar[r]\ar[d]^{0}  & 0\ar[r]\ar[d] & \cdots\\
X^{a}[-a]\quad \ar[r] & 0\ar[r] & X^{a}\ar[r] & 0\ar[r] & \cdots \ar[r] & 0\ar[r]  & 0\ar[r] & \cdots}$$
That is we have the exact sequence of complexes
$$\xymatrix{0\ar[r] & \sigma^{>a}(X^{\bullet})\ar[r] & X^{\bullet}\ar[r] & X^{a}[-a]\ar[r] & 0}$$
where $X^{a}[-a]$ is the complex concentrated in degree $a$, such that in degree $a$ has the term $X^{a}$ and zero elsewhere. Hence we have a triangle in the derived category $\mathbf{D}^{b}(\mathcal{A})$:
$$\xymatrix{\sigma^{>a}(X^{\bullet})\ar[r] &  X^{\bullet}\ar[r] & X^{a}[-a]\ar[r] & \sigma^{>a}(X^{\bullet})[1]}$$
\end{proof}

\begin{lemma}\label{Lmandaperf}
Let $\mathcal{A}$ and $\mathcal{B}$ be abelian categories with enough projectives. Let $F:\mathcal{A}\longrightarrow \mathcal{B}$ a right exact functor. Suppose that we have left derived funtor $L^{b}(F):D^{b}(\mathcal{A})\longrightarrow D^{b}(\mathcal{B})$. If $F$ preserve projectives, then 
$$L^{b}(F)(P^{\bullet})\in \mathrm{Perf}(\mathcal{B})$$
for all $P^{\bullet}\in \mathrm{Perf}(\mathcal{A})$.
\end{lemma}

\begin{proof}
Consider
$$\xymatrix{P^{\bullet}:\quad \ar[r] & 0\ar[r]  & P^{a}\ar[r] & P^{a+1}\ar[r] & \cdots \ar[r] & P^{b}\ar[r]  & 0\ar[r] & \cdots}$$
where each $P^{i}$ is a projective object in $\mathcal{A}$. We proceed, by induction on the length of the complex $n:=b-a$. If $n=0$,
we have that $P^{\bullet}$ is of the form $P[k]$ for some projective object $P\in\mathcal{A}$ and $k\in\mathbb{Z}$. Hence $L^{b}(F)(P^{\bullet})=L^{b}(F)(P[k])=F(P)[k]\in \mathrm{Perf}(\mathcal{B})$ since $F(P)$ is a projective object in $\mathcal{B}$.\\
Consider $P^{\bullet}$ with lenght $n=b-a\geq 1$ and its stupid truncation $\sigma^{>a}(P^{\bullet})$. Hence, by Lemma \ref{triangululocan}, we have a triangle in the derived category $\mathbf{D}^{b}(\mathcal{A})$:
$$\xymatrix{\sigma^{>a}(P^{\bullet})\ar[r] &  P^{\bullet}\ar[r] & P^{a}[-a]\ar[r] & \sigma^{>a}(P^{\bullet})[1]}$$
where $\sigma^{>a}(P^{\bullet})$ is a perfect complex  with lenght  $n-1=b-(a+1)$ and $P^{a}[-a]$ with lenght $0$.
Since $L^{b}(F)$ is triangulated functor, we have the triangle in $\mathbf{D}^{b}(\mathcal{B})$:
$$\xymatrix{L^{b}(F)(\sigma^{>a}(P^{\bullet}))\ar[r] &  L^{b}(F)(P^{\bullet})\ar[r] & L^{b}(F)(P^{a}[-a])\ar[r] & L^{b}(F)(\sigma^{>a}(P^{\bullet}))[1]}.$$
By induction hypothesis we have that $L^{b}(F)(\sigma^{>a}(P^{\bullet})), L^{b}(F)(P^{a}[-a])\in \mathrm{Perf}(\mathcal{B})$ and  since $\mathrm{Perf}(\mathcal{B})$ is a triangulated subcategory we conclude that $ L^{b}(F)(P^{\bullet})\in \mathrm{Perf}(\mathcal{B})$.
\end{proof}

\begin{corollary}
Let $\mathcal{A}$ and $\mathcal{B}$ be abelian categories with enough projectives. Let $F:\mathcal{A}\longrightarrow \mathcal{B}$ be a right exact functor, and let $L^{b}(F):D^{b}(\mathcal{A})\longrightarrow D^{b}(\mathcal{B})$ be its left derived functor. If $F$ preserves projectives, then it induces a functor
$$\overline{L^{b}(F)}:\frac{\mathbf{D}^{b}(\mathcal{A})}{\mathrm{Perf}(\mathcal{A})}\longrightarrow \frac{\mathbf{D}^{b}(\mathcal{B})}{\mathrm{Perf}(\mathcal{B})}.$$
\end{corollary}

\begin{lemma}\label{perfectideal}
Let $\mathcal{C}$ be a $K$-category and $M\in \mathrm{Mod}(\mathcal{C}^{e})$ of finite projective dimension in $\mathrm{Mod}(\mathcal{C}^{e})$. Hence $L_{I}^{-}\mathbb{F}(M,Y^{\bullet})\in \mathrm{Perf}(\mathrm{Mod}(\mathcal{C}))$ for $Y^{\bullet}\in K^{b}(\mathrm{Mod}(\mathcal{C}))$.
\end{lemma}
\begin{proof}
Let $P^{\bullet}\longrightarrow M$ be a projective resolution of $M$ in  $\mathrm{Mod}(\mathcal{C}^{e})$, that is, we have the exact sequence 
$$\xymatrix{0\ar[r] & P^{n}\ar[r] & P^{n-1}\ar[r] & \cdots\ar[r] & P^{1}\ar[r] & P^{0}\ar[r] & M\ar[r] & 0}$$
where each $P^{i}\in \mathrm{Proj}(\mathrm{Mod}(\mathcal{C}^{e}))$.
By definition we have that

$$L_{I}^{-}\mathbb{F}(M,Y^{\bullet})=P^{\bullet}\boxtimes_{\mathcal{C}}Y^{\bullet}.$$
where $(P^{\bullet}\boxtimes_{\mathcal{C}}Y^{\bullet})^{i}:=\bigoplus_{p+q=i}P^{p}\boxtimes_{\mathcal{C}}Y^{q}$.
Since $P^{p}$ is projective in $\mathrm{Mod}(\mathcal{C}^{e})$ and $Y^{q}(C)$ is projective in $\mathrm{Mod}(K)$ for all $C\in \mathcal{C}$, since $K$ is a field (recall $Y^{q}:\mathcal{C}\longrightarrow \mathrm{Mod}(K)$). By Proposition 11.6 (i) in \cite{Mitchelring}, we have that $P^{p}\boxtimes_{\mathcal{C}}Y^{q}$ is projective in $\mathrm{Mod}(\mathcal{C})$ and hence  $(P^{\bullet}\boxtimes_{\mathcal{C}}Y^{\bullet})^{i}:=\bigoplus_{p+q=i}P^{p}\boxtimes_{\mathcal{C}}Y^{q}$ is projective in $\mathrm{Mod}(\mathcal{C})$. Now, since $Y^{\bullet}$ and $P^{\bullet}$ are bounded complexes we have that $P^{\bullet}\boxtimes_{\mathcal{C}}Y^{\bullet}$ is a bounded complex and hence  $L_{I}^{-}\mathbb{F}(M,Y^{\bullet})\in \mathrm{Perf}(\mathrm{Mod}(\mathcal{C}))$.
\end{proof}

\begin{lemma}\label{pirestricperf}
Let $\mathcal{C}$ be a $k$-category and $I$ be a strongly idempotent ideal which has a finite projective dimension in $\mathrm{Mod}(\mathcal{C}^{e})$. Then the derived functor
$$L(\pi_{\ast})=\pi_{\ast}: \mathbf{D}^{b}(\mathrm{Mod}(\mathcal{C}/I))\longrightarrow  \mathbf{D}^{b}(\mathrm{Mod}(\mathcal{C}))$$
sends perfect complexes into perfect complexes.
\end{lemma}
\begin{proof}
Consider the exact sequence in $\mathrm{Mod}(\mathcal{C}^{e})$:
$$\xymatrix{0\ar[r] & I\ar[r] & \mathcal{C}\ar[r] & H\ar[r] & 0}$$
where  $H:=(\mathcal{C}/I)\circ (\pi^{op}\otimes \pi)\in  \mathrm{Mod}(\mathcal{C}^{e})$.\\
Suppose that $P^{\bullet}$ is a  finite projective resolution of $I$ in $\mathrm{Mod}(\mathcal{C}^{e})$:
$$\scalebox{0.9}{\xymatrix{P^{\bullet}: 0\ar[r] & P_{n}(-,-)\ar[r] & \cdots\ar[r] & P_{1}(-,-)\ar[r] & P_{0}(-,-)\ar[r] & I(-,-)\ar[r] & 0}}.$$
By Lemma \ref{CIfinitedimension},  for each $C\in \mathcal{C}$ we have the projective resolution of $\mathcal{C}(C,-)/I(C,-)$: 
$$\scalebox{0.9}{\xymatrix{0\ar[r] & P_{n}(C,-)\ar[r] & \cdots\ar[r] & P_{0}(C,-)\ar[r] &  \mathcal{C}(C,-)\ar[r] &\mathcal{C}(C,-)/I(C,-)\ar[r] & 0 }}.$$
Hence, we have the following exact sequence
$$\scalebox{0.8}{\xymatrix{0\ar[r] & \bigoplus_{i\in I}P_{n}(C_{i},-)\ar[r] & \cdots\ar[r] & \bigoplus_{i\in I}P_{0}(C_{i},-)\ar[r] &  \bigoplus_{i\in I}\mathcal{C}(C_{i},-)\ar[r] & \bigoplus_{i\in I}\frac{\mathcal{C}(C_{i},-)}{I(C_{i},-)}\ar[r] & 0 }}$$
for every set $I$ (each $C_{i}$ can be repeated several times).\\
Let $Q^{\bullet}$ be the complex
$$\scalebox{1}{\xymatrix{0\ar[r] & \bigoplus_{i\in I}P_{n}(C_{i},-)\ar[r] & \cdots\ar[r] & \bigoplus_{i\in I}P_{0}(C_{i},-)\ar[r] &  \bigoplus_{i\in I}\mathcal{C}(C_{i},-) }}$$
Hence, $\bigoplus_{i\in I}\frac{\mathcal{C}(C_{i},-)}{I(C_{i},-)}$ is quasi-isomorphic to $Q^{\bullet}$. 
Consider the funtor
$$\pi_{\ast}: \mathbf{D}^{b}(\mathrm{Mod}(\mathcal{C}/I))\longrightarrow  \mathbf{D}^{b}(\mathrm{Mod}(\mathcal{C})).$$
Thus, we have that
$$\pi_{\ast}\Big(\bigoplus_{i\in  I}(\mathcal{C}/I)(C_{i},-)\Big)=\bigoplus_{i\in I}\frac{\mathcal{C}(C_{i},-)}{I(C_{i},-)}$$  is perfect in $\mathbf{D}^{b}(\mathrm{Mod}(\mathcal{C}))$.\\
 Let $P$ be an arbitrary projective object in $\mathrm{Mod}(\mathcal{C}/I)$, hence $P$ is a direct summand of $\bigoplus_{i\in  I}(\mathcal{C}/I)(C_{i},-)$  for some set $I$. Since $\bigoplus_{i\in I}\frac{\mathcal{C}(C_{i},-)}{I(C_{i},-)}$ is perfect we conclude that $\pi_{\ast}(P)$ is a direct summand of a perfect complex and hence $\pi_{\ast}(P)$ is a perfect complex in  $\mathbf{D}^{b}(\mathrm{Mod}(\mathcal{C}))$ ($\mathrm{Perf}(\mathrm{Mod}(\mathcal{C}))$ is a thick triangulated subcategory of $\mathbf{D}^{b}(\mathrm{Mod}(\mathcal{C}))$).\\
 Now, we will show that if $P^{\bullet}$ is a perfect complex in $\mathbf{D}^{b}(\mathrm{Mod}(\mathcal{C}/I))$, then 
 $\pi_{\ast}(P^{\bullet})$ is a perfect complex in $\mathbf{D}^{b}(\mathrm{Mod}(\mathcal{C}))$.\\
 Let $P^{\bullet}$ be a  perfect complex in  $\mathbf{D}^{b}(\mathrm{Mod}(\mathcal{C}/I))$
 $$\xymatrix{P^{\bullet}:\quad \ar[r] & 0\ar[r]  & P^{a}\ar[r] & P^{a+1}\ar[r] & \cdots \ar[r] & P^{b}\ar[r]  & 0\ar[r] & \cdots}$$ 
 The proof is by induction on the length $n=b-a$. If $n=0$, then $P^{\bullet}=P[k]$ for some projective object $P\in \mathrm{Mod}(\mathcal{C}/I)$ and $k\in\mathbb{Z}$. Hence
 $$\pi_{\ast}(P^{\bullet})=\pi_{\ast}(P[k])=\pi_{\ast}(P)[k]$$
 is a perfect complex in $\mathbf{D}^{b}(\mathrm{Mod}(\mathcal{C}/I))$ by the above discussion.\\
 Consider $P^{\bullet}$ with length  $n=b-a>1$ and its stupid truncation $\sigma^{>a}(P^{\bullet})$. By Lemma \ref{triangululocan}, we have  a triangle in the derived category $\mathbf{D}^{b}(\mathrm{Mod}(\mathcal{C}/I))$:
$$\xymatrix{\sigma^{>a}(P^{\bullet})\ar[r] &  P^{\bullet}\ar[r] & P^{a}[-a]\ar[r] & \sigma^{>a}(P^{\bullet})[1]}$$
where $\sigma^{>a}(P^{\bullet})$ is a perfect complex  with length $b-(a+1)=n-1$ and $P^{a}[-a]$ is a perfect with length $0$. Since $\pi_{\ast}$ is triangulated functor, we have the triangle in $\mathbf{D}^{b}(\mathrm{Mod}(\mathcal{C}))$:
$$\xymatrix{\pi_{\ast}(\sigma^{>a}(P^{\bullet}))\ar[r] &  \pi_{\ast}(P^{\bullet})\ar[r] & \pi_{\ast}(P^{a}[-a])\ar[r] & \pi_{\ast}(\sigma^{>a}(P^{\bullet}))[1]}$$
By induction hypothesis we have that $\pi_{\ast}(\sigma^{>a}(P^{\bullet})), \pi_{\ast}(P^{a}[-a])\in \mathrm{Perf}(\mathrm{Mod}(\mathcal{C}))$ and  since $\mathrm{Perf}(\mathrm{Mod}(\mathcal{C}))$ is a triangulated subcategory we conclude that $ \pi_{\ast}(P^{\bullet})\in \mathrm{Perf}(\mathrm{Mod}(\mathcal{C}))$.
\end{proof}

\begin{lemma}$\textnormal{\cite[Lemma 1.2]{Orlov1}}$\label{lema2.2Xiao}
Let $F:\mathcal{T}\longrightarrow \mathcal{T}'$ be a triangulated functor which has a right adjoint $G$. Assume that $\mathcal{N}\subseteq \mathcal{T}$ and $\mathcal{N}'\subseteq \mathcal{T}'$ are triangulated subcategories satisfying that $F(\mathcal{N})\subseteq \mathcal{N}'$ and $G(\mathcal{N}')\subseteq \mathcal{N}$. Then, the induced functor $\overline{F}:\mathcal{T}/\mathcal{N}\longrightarrow \mathcal{T}'/\mathcal{N}'$ has a right adjoint  $\overline{G}:\mathcal{T}'/\mathcal{N}'\longrightarrow \mathcal{T}/\mathcal{N}$. Moreover, if $G$ is full and faithfull, so is $\overline{G}$.
\end{lemma}
\begin{proof}
For a proof see \cite[Lemma 2.2]{Xiao}.
\end{proof}

\begin{lemma}\label{lema2.1Xiao}
Let $F:\mathcal{T}\longrightarrow \mathcal{T}'$ be a triangulated functor which admits a full and faithful right adjoint $G$. Then $F$ induces a triangle equivalence $\mathcal{T}/\mathrm{Ker}(F)\cong \mathcal{T}'$.
\end{lemma}
\begin{proof}
For a proof see \cite[Lemma 2.1]{Xiao} 
\end{proof}

%A $\textbf{Krull}$-$\textbf{Schmidt}$ category  is an additive category such that each object decomposes into a finite direct sum of indecomposable objects with local endomorphism rings.\\

\begin{theorem}\label{teoremaprincipal}
Let $\mathcal{C}$ be a $K$-category and $I$ be a strongly idempotent ideal which has a finite projective dimension in $\mathrm{Mod}(\mathcal{C}^{e})$. Then there exists a singular equivalence between $\mathcal{C}$ and $\mathcal{C}/I$.
\end{theorem}
\begin{proof}
Let $\pi:\mathcal{C}\longrightarrow \mathcal{C}/I$ be the canonical functor. Recall that we have a triple adjoint
$$\xymatrix{\mathrm{Mod}(\mathcal{C}/\mathcal{I})\ar[rr]|{\pi_{\ast}}  &  &\mathrm{Mod}(\mathcal{C})\ar@<-2ex>[ll]_{\pi^{\ast}}\ar@<2ex>[ll]^{\pi^{!}} }$$
where $(\pi^{\ast},\pi_{\ast})$ and $(\pi_{\ast},\pi^{!})$ are adjoint pairs, $\pi^{!}:=\mathcal{C}(\frac{\mathcal{C}}{\mathcal{I}},-)$ and $\pi^{\ast}:=\frac{\mathcal{C}}{\mathcal{I}}\otimes_{\mathcal{C}}$ (see for example \cite[Proposition 3.11]{RSS}).\\
Since $\pi_{\ast}$ is exact  we conclude that $\pi_{\ast}$ has finite right cohomological dimension. By Lemma  \ref{tensorfiniecomo}, we have that $\pi^{\ast}:=\frac{\mathcal{C}}{\mathcal{I}}\otimes_{\mathcal{C}}$ has finite left cohomological dimension. Hence, by Lemma \ref{lema1Cline-Parshall}, we have that the left derived funtor 
$$\mathcal{C}/I\otimes_{\mathcal{C}}^{L}-:=L_{II}^{-}\mathbb{F}(H_{1},-):\mathbf{D}^{b}(\mathrm{Mod}(\mathcal{C}))\longrightarrow \mathbf{D}^{b}(\mathrm{Mod}(\mathcal{C}/I))$$ is left adjoint to  $$\pi_{\ast}=L(\pi_{\ast}): \mathbf{D}^{b}(\mathrm{Mod}(\mathcal{C}/I))\longrightarrow  \mathbf{D}^{b}(\mathrm{Mod}(\mathcal{C})).$$
By  Lemma \ref{pirestricperf}, we have that $\pi_{\ast}$ send perfect complexes into perfect complexes and thus we obtain the induced functor
$$\overline{\pi_{\ast}}:\frac{\mathbf{D}^{b}(\mathrm{Mod}(\mathcal{C}/I))}{\mathrm{Perf}(\mathrm{Mod}(\mathcal{C}/I))}\longrightarrow \frac{\mathbf{D}^{b}(\mathrm{Mod}(\mathcal{C}))}{\mathrm{Perf}(\mathrm{Mod}(\mathcal{C}))}.$$
Now, since $\mathcal{C}/I\otimes_{\mathcal{C}}\mathcal{C}(C,-)=(\mathcal{C}/I)(C,-)$ (see Prop. 3.5 in p. 793 in \cite{LGOS2}), we conclude that $\mathcal{C}/I\otimes_{\mathcal{C}}-$ preserve projective objects. Thus, by Lemma \ref{Lmandaperf}, we get that $\mathcal{C}/I\otimes_{\mathcal{C}}^{L}-$ send perfect complexes into perfect complexes. Hence we have the induced functor
$$G=\overline{\mathcal{C}/I\otimes_{\mathcal{C}}^{L}-}:\frac{\mathbf{D}^{b}(\mathrm{Mod}(\mathcal{C}))}{\mathrm{Perf}(\mathrm{Mod}(\mathcal{C}))}\longrightarrow \frac{\mathbf{D}^{b}(\mathrm{Mod}(\mathcal{C}/I))}{\mathrm{Perf}(\mathrm{Mod}(\mathcal{C}/I))}.$$
Since $\pi:\mathcal{C}\longrightarrow \mathcal{C}/I$ is a homological epimorphism (see Proposition \ref{caractidem2}), we conclude that the functor $\pi_{\ast}: \mathbf{D}^{b}(\mathrm{Mod}(\mathcal{C}/I))\longrightarrow  \mathbf{D}^{b}(\mathrm{Mod}(\mathcal{C}))$ is full and faithful and hence by Lemma  \ref{lema2.2Xiao} we have that $\overline{\pi_{\ast}}$ is full and faithful. That is, we have an adjoint pair $\Big(\overline{\mathcal{C}/I\otimes_{\mathcal{C}}^{L}-}, \overline{\pi_{\ast}}\Big)$ where $\overline{\pi_{\ast}}$ is full and faithful. Now, by Lemma \ref{lema2.1Xiao}, we have that $G$ induces an equivalence
$$\widehat{G}:\mathbf{D}_{sg}(\mathrm{Mod}(\mathcal{C}))/\mathrm{Ker}(G)\longrightarrow \mathbf{D}_{sg}(\mathrm{Mod}(\mathcal{C}/I)).$$
Let us see that $\mathrm{Ker}(G)=0$.\\
 Indeed, let $Y^{\bullet}\in \mathbf{D}^{b}(\mathrm{Mod}(\mathcal{C}))$ such that
$G(Y^{\bullet})=\overline{\mathcal{C}/I\otimes_{\mathcal{C}}^{L}Y^{\bullet}}=0$. This implies that $\mathcal{C}/I\otimes_{\mathcal{C}}^{L}Y^{\bullet}=L_{II}^{-}\mathbb{F}(H_{1},Y^{\bullet})$ is a perfect complex in $\mathbf{D}^{b}(\mathrm{Mod}(\mathcal{C}/I))$.\\
%%%%%%%%%%%%%%%%%%%%%%%%%%%%%%%%%%%%%%%%%%%%%%%%%%%%%%%%
Consider the left derived functor (see Proposition \ref{funtorderivadotensor})
 $$L_{I}^{-}\mathbb{F}(-,Y^{\bullet}):\mathbf{D}^{-}\Big(\mathrm{Mod}(\mathcal{C}^{op}\otimes_K \mathcal{C})\Big)\longrightarrow \mathbf{D}^{-}(\mathrm{Mod}(\mathcal{C})).$$
 Recall that we have the following exact sequence in $\mathrm{Mod}(\mathcal{C}^{e})=\mathrm{Mod}(\mathcal{C}^{op}\otimes_K \mathcal{C})$:
 $$\xymatrix{0\ar[r] & I\ar[r] & \mathcal{C}\ar[r] & H\ar[r] & 0}$$
 where $H=(\mathcal{C}/I)\circ (\pi^{op}\otimes \pi)\in \mathrm{Mod}(\mathcal{C}^{op}\otimes_K \mathcal{C})$. Hence, we get a triangle in $\mathbf{D}^{-}(\mathrm{Mod}(\mathcal{C}^{e}))$:
 $$\xymatrix{I\ar[r] & \mathcal{C}\ar[r] & H\ar[r] & I[1].}$$
 Thus, we obtain a triangle in $\mathbf{D}^{-}(\mathrm{Mod}(\mathcal{C}))$
  $$(\ast):\xymatrix{L_{I}^{-}\mathbb{F}(I,Y^{\bullet})\ar[r] & L_{I}^{-}\mathbb{F}(\mathcal{C},Y^{\bullet})\ar[r] & L_{I}^{-}\mathbb{F}(H,Y^{\bullet})\ar[r] & L_{I}^{-}\mathbb{F}(I,Y^{\bullet})[1]}.$$
On the other hand, by  Corollary \ref{compodosfuntor}, we have that 
$$L_{II}^{-}\mathbb{F}(H,-)=L(\pi_{\ast})\circ L_{II}^{-}\mathbb{F}(H_{1},-)=\pi_{\ast}\circ (\mathcal{C}/I\otimes_{\mathcal{C}}^{L}-) $$
where $H_{1}=(\mathcal{C}/I)\circ (\pi^{op}\otimes 1)\in \mathrm{Mod}((\mathcal{C}/I)^{op}\otimes_K \mathcal{C})$.\\
Hence, for $Y^{\bullet}\in \mathbf{D}^{b}(\mathrm{Mod}(\mathcal{C}))$, by Proposition \ref{funtorderivadotensor}(c), we have that
$$\pi_{\ast}\Big( \mathcal{C}/I\otimes_{\mathcal{C}}^{L}Y^{\bullet}\Big)=L_{II}^{-}\mathbb{F}(H,Y^{\bullet})=L^{-}\mathbb{F}(H,Y^{\bullet})=L_{I}^{-}\mathbb{F}(H,Y^{\bullet}).$$
Since $\mathcal{C}/I\otimes_{\mathcal{C}}^{L}Y^{\bullet}$ is a perfect complex and $\pi_{\ast}$ preserves perfect complexes, we get that $L_{I}^{-}\mathbb{F}(H,Y^{\bullet})$ is a perfect complex.\\
On the other hand,  by Proposition \ref{funtorderivadotensor}, we have that
 $$L_{I}^{-}\mathbb{F}(\mathcal{C},Y^{\bullet})=L_{II}^{-}\mathbb{F}(\mathcal{C},Y^{\bullet}).$$
 In order to compute $L_{II}^{-}\mathbb{F}(\mathcal{C},Y^{\bullet})$ we consider $\alpha:P^{\bullet}\longrightarrow Y^{\bullet}$ be a quasi-isomorphism in $\mathbf{K}^{-}\Big(\mathrm{Mod}(\mathcal{C})\Big)$ where $P^{\bullet}$ is a complex of projective modules. Hence, $L_{II}^{-}\mathbb{F}(\mathcal{C},Y^{\bullet})=\mathcal{C}\boxtimes_{\mathcal{C}}P^{\bullet}=P^{\bullet}$.
Thus, we have an isomorphism $L_{II}^{-}\mathbb{F}(\mathcal{C},Y^{\bullet})\simeq P^{\bullet}\simeq Y^{\bullet}$ in $D^{-}\Big(\mathrm{Mod}(\mathcal{C})\Big)$.
 Therefore,
$$L_{I}^{-}\mathbb{F}(\mathcal{C},Y^{\bullet})\simeq Y^{\bullet}.$$
By  Lemma \ref{perfectideal}, we have that  $L_{I}^{-}\mathbb{F}(I,Y^{\bullet})$ is a perfect  complex.\\
Then, we have that in the triangle $(\ast)$ the first and third term are perfect complexes, hence we conclude that the middle term $L_{I}^{-}\mathbb{F}(\mathcal{C},Y^{\bullet})\simeq Y^{\bullet}$ also is a perfect complex.\\
Hence $\mathrm{Ker}(G)=0$ and hence we have the desired equivalence.
\end{proof}

\section{Application to triangular matrix categories}
We consider the triangular matrix category  $\Lambda:=\left[\begin{smallmatrix}
\mathcal{T} & 0 \\ 
M & \mathcal{U}
\end{smallmatrix}\right]$ constructed in \cite{LeOS} and defined as follows.
%%%%%%%%%%%%%%%%%%%%%%%%%%
\begin{definition}$\textnormal{\cite[Definition 3.5]{LeOS}}$ \label{defitrinagularmat}
Let $\mathcal{U}$ and $\mathcal{T}$ be two $K$-categories, and consider an additive $K$-functor  $M$ from the tensor product category  $ \mathcal{T}^{op}\otimes_K\mathcal{U}$ to the category $\mathrm{Mod}(K)$.
The \textbf{triangular matrix category}
$\Lambda=\left[ \begin{smallmatrix}
\mathcal{T} & 0 \\ M & \mathcal{U}
\end{smallmatrix}\right]$ is defined as below.
\begin{enumerate}
\item [(a)] The class of objects of this category are matrices $ \left[
\begin{smallmatrix}
T & 0 \\ M & U
\end{smallmatrix}\right]  $ with $ T\in \mathcal{T} $ and $ U\in \mathcal{U} $.

\item [(b)] For objects
$\left[ \begin{smallmatrix}
T & 0 \\
M & U
\end{smallmatrix} \right] ,  \left[ \begin{smallmatrix}
T' & 0 \\
M & U'
\end{smallmatrix} \right]\in\Lambda$, we define $$\Lambda\left (\left[ \begin{smallmatrix}
T & 0 \\
M & U
\end{smallmatrix} \right] ,  \left[ \begin{smallmatrix}
T' & 0 \\
M & U'
\end{smallmatrix} \right]  \right)  := \left[ \begin{smallmatrix}
\mathcal{T}(T,T') & 0 \\
M(T,U') & \mathcal{U}(U,U')
\end{smallmatrix} \right].$$
\end{enumerate}
The composition is given by
\begin{eqnarray*}
\circ&:&\left[  \begin{smallmatrix}
{\mathcal{T}}(T',T'') & 0 \\
M(T',U'') & {\mathcal{U}}(U',U'')
\end{smallmatrix}  \right] \times \left[
\begin{smallmatrix}
{\mathcal{T}}(T,T') & 0 \\
M(T,U') & {\mathcal{U}}(U,U')
\end{smallmatrix} \right]\longrightarrow\left[
\begin{smallmatrix}
{\mathcal{T}}(T,T'') & 0 \\
M(T,U'') & {\mathcal{U}}(U,U'')\end{smallmatrix} \right] \\
&& \left( \left[ \begin{smallmatrix}
t_{2} & 0 \\
m_{2} & u_{2}
\end{smallmatrix} \right], \left[
\begin{smallmatrix}
t_{1} & 0 \\
m_{1} & u_{1}
\end{smallmatrix} \right]\right)\longmapsto\left[
\begin{smallmatrix}
t_{2}\circ t_{1} & 0 \\
m_{2}\bullet t_{1}+u_{2}\bullet m_{1} & u_{2}\circ u_{1}
\end{smallmatrix} \right].
\end{eqnarray*}
\end{definition}
We recall that $ m_{2}\bullet t_{1}:=M( t_{1}^{op}\otimes 1_{U''})(m_{2})$ and
$u_{2}\bullet m_{1}=M(1_T\otimes u_{2})(m_{1})$.
Thus, $\Lambda$ is clearly a $K$-category since
$\mathcal{T} $ and $\mathcal{U}$ are $K$-categories and $M(T,U')$ is a $K$-module.\\
We define a functor $\Phi:\Lambda\longrightarrow \mathcal{U}$
as follows: $\Phi\Big(\left[\begin{smallmatrix}
 T & 0 \\ 
M & U
\end{smallmatrix}\right]\Big):=U$ and for 
 $\left[\begin{smallmatrix}
\alpha & 0 \\ 
m & \beta
\end{smallmatrix}\right]:\left[\begin{smallmatrix}
 T & 0 \\ 
M & U
\end{smallmatrix}\right]\longrightarrow \left[\begin{smallmatrix}
T' & 0 \\ 
M & U'
\end{smallmatrix}\right]$ we set $\Phi\Big(\left[\begin{smallmatrix}
\alpha & 0 \\ 
m & \beta
\end{smallmatrix}\right]\Big)=\beta$.\\
For simplicity, we will write
$\mathfrak{M}=\left[\begin{smallmatrix}
 T & 0 \\ 
M & U
\end{smallmatrix}\right]\in \Lambda$.

\begin{lemma}\label{exacseuqnimpor}
 There exists an exact sequence
in $\mathrm{Mod}(\Lambda^{e})$
$$\xymatrix{0\ar[r] & \mathcal{I}\ar[r] & \Lambda\ar[rr]^(.3){\Gamma(\Phi)} & & \mathcal{U}(-,-)\circ (\Phi^{op}\otimes\Phi)\ar[r] & 0,}$$
where for objects $\mathfrak{M}=\left[\begin{smallmatrix}
 T & 0 \\ 
M & U
\end{smallmatrix}\right]$ and $\mathfrak{M}'=\left[\begin{smallmatrix}
 T' & 0 \\ 
M & U'
\end{smallmatrix}\right]$ in $\Lambda$ the ideal $\mathcal{I}$ is given as  $\mathcal{I}\big(\mathfrak{M},\mathfrak{M}'\big)=\mathrm{Ker}\left([\Gamma (\Phi)]_{\big(\mathfrak{M},\mathfrak{M}'\big)}\right)=\left[\begin{smallmatrix}
 \mathcal{T}(T,T') & 0 \\ 
M(T,U') &  0
\end{smallmatrix}\right]$.
\end{lemma}
\begin{proof}
It is straightforward.
\end{proof}

\begin{remark}\label{projectivida}
We can see that $\mathcal{I}\Big({\left[\begin{smallmatrix}
 T & 0 \\ 
M & U
\end{smallmatrix}\right]},-\Big)\simeq {\Lambda}\Big({\left[\begin{smallmatrix}
 T & 0 \\ 
M & 0
\end{smallmatrix}\right]},-\Big)$, and, from this, it follows that $\mathcal{I}\Big({\left[\begin{smallmatrix}
 T & 0 \\ 
M & U
\end{smallmatrix}\right]},-\Big)$ is projective in $\mathrm{Mod}(\Lambda)$.
\end{remark}

\begin{proposition}\label{proyepihomomat}
The functor $\Gamma(\Phi):\Lambda\longrightarrow  \mathcal{U}(-,-)\circ (\Phi^{op}\otimes\Phi)$ is a homological epimorphism.
\end{proposition}

\begin{proof}
We have an epimorphism $\Phi:\Lambda\longrightarrow \mathcal{U}$ and an exact sequence
in $\mathrm{Mod}(\Lambda^{e})$
$$\xymatrix{0\ar[r] & \mathcal{I}\ar[r] & \Lambda\ar[rr]^(.3){\Gamma(\Phi)} & & \mathcal{U}(-,-)\circ (\Phi^{op}\otimes\Phi)\ar[r] & 0.}$$
We notice that $\mathcal{I}$ is an ideal of $\Lambda$ and $\mathcal{U}\simeq \Lambda/\mathcal{I}$.
By Remark  \ref{projectivida}, we get that $\mathcal{I}(\mathfrak{M},-)$ is projective in $\mathrm{Mod}(\Lambda)$ for all $\mathfrak{M}\in \Lambda$. Hence by Proposition \ref{proyepihomo}, we have that $\mathcal{I}$ is strongly idempotent.
\end{proof}

\subsection{One point extension category}

In this section, $\mathcal{U}$ will denote a $K$-category and $M:\mathcal{U}\longrightarrow \mathrm{Mod}(K)$ a $K$-functor. We consider $\mathcal{C}_{K}$ the $K$-category with only one object, namely $\mathrm{obj}(\mathcal{C}_{K}):=\{\ast\}$, and the canonical isomorphism  $\Delta:  \mathcal{C}_{K}^{op}\otimes\mathcal{U}\longrightarrow \mathcal{U}$.
Then, we get  $\underline{M}: \mathcal{C}_{K}^{op}\otimes \mathcal{U}\longrightarrow \mathrm{Mod}(K)$ given as
$\underline{M}:=M\circ\Delta$.
Hence, we can construct the matrix category $\Lambda:=\left[ \begin{smallmatrix}
\mathcal{C}_{K} & 0 \\ 
\underline{M} & \mathcal{U}
\end{smallmatrix}\right].$ This matrix category is called the $\textbf{one-point extension category}$ because it is a generalization of the well-known construction of the one point-extension algebra; see for example page 71 in \cite{AusBook}. \\\
For the case  $\Lambda:=\left[ \begin{smallmatrix}
\mathcal{C}_{K} & 0 \\ 
\underline{M} & \mathcal{U}
\end{smallmatrix}\right]$, the ideal $\mathcal{I}$ in the Lemma \ref{exacseuqnimpor} si given as follows:  for objects $\mathfrak{M}=\left[\begin{smallmatrix}
\ast & 0 \\ 
\underline{M} & U
\end{smallmatrix}\right]$ and $\mathfrak{M}'=\left[\begin{smallmatrix}
\ast & 0 \\ 
\underline{M} & U'
\end{smallmatrix}\right]$ in $\Lambda$ we have that $\mathcal{I}\big(\mathfrak{M},\mathfrak{M}'\big)=\left[\begin{smallmatrix}
 \mathcal{C}_{K}(\ast,\ast) & 0 \\ 
\underline{M}(\ast,U') &  0
\end{smallmatrix}\right]=\left[\begin{smallmatrix} K & 0 \\ 
M(U') &  0
\end{smallmatrix}\right]$.

\begin{lemma}\label{bimoduloproy}
Consider the following object $\mathfrak{N}= \left[\begin{smallmatrix} \ast & 0 \\ 
\underline{M} &  0
\end{smallmatrix}\right]\in \Lambda$ and $\Lambda^{e}\Big((\mathfrak{N},\mathfrak{N}),(-,-)\Big)\in \mathrm{Mod}(\Lambda^{e})$. Then $\mathcal{I}(-,-)\simeq\Lambda^{e}\Big((\mathfrak{N},\mathfrak{N}),(-,-)\Big)$, in particular $\mathcal{I}$ is projective in $\mathrm{Mod}(\Lambda^{e})$.
\end{lemma}
\begin{proof}
For  
$$f:=\left[\begin{smallmatrix}
\lambda & 0 \\ 
m & \beta
\end{smallmatrix}\right]:\left[\begin{smallmatrix}
 \ast & 0 \\ 
\underline{M} & U_{1}
\end{smallmatrix}\right]=\mathfrak{M}_{1}\longrightarrow \left[\begin{smallmatrix}
\ast & 0 \\ 
\underline{M} & U_{2}
\end{smallmatrix}\right]=\mathfrak{M}_{2}, \text{ and}$$ 
$$g:=\left[\begin{smallmatrix}
\lambda' & 0 \\ 
m' & \beta'
\end{smallmatrix}\right]:\left[\begin{smallmatrix}
 \ast & 0 \\ 
\underline{M} & U_{3}
\end{smallmatrix}\right]=\mathfrak{M}_{3}\longrightarrow \left[\begin{smallmatrix}
\ast & 0 \\ 
\underline{M} & U_{4}
\end{smallmatrix}\right]=\mathfrak{M}_{4},$$  we have $f^{op}\otimes g:(\mathfrak{M}_{2},\mathfrak{M}_{3})\longrightarrow (\mathfrak{M}_{1},\mathfrak{M}_{4})$ a morphism in $\Lambda^{e}$ and hence we have a morphism of abelian groups

$$\mathcal{I}(f^{op}\otimes g):\mathcal{I}(\mathfrak{M}_{2},\mathfrak{M}_{3})=\left[\begin{smallmatrix} K & 0 \\ 
M(U_{3}) &  0
\end{smallmatrix}\right]\longrightarrow \mathcal{I}(\mathfrak{M}_{1},\mathfrak{M}_{4})=\left[\begin{smallmatrix} K & 0 \\ 
M(U_{4}) &  0
\end{smallmatrix}\right],$$
where
$$\mathcal{I}(f^{op}\otimes g)=\Lambda(f^{op}\otimes g)|_{\mathcal{I}(\mathfrak{M}_{2},\mathfrak{M}_{3})}$$
Recall that for $f^{op}\otimes g$ we have that the morphism $$\Lambda(f^{op}\otimes g):\Lambda(\mathfrak{M}_{2},\mathfrak{M}_{3})=
\left[\begin{smallmatrix}  \mathcal{C}_{K}(\ast,\ast) & 0 \\ 
M(U_{3}) &  \mathcal{U}(U_{2},U_{3})
\end{smallmatrix}\right]\longrightarrow \Lambda(\mathfrak{M}_{1},\mathfrak{M}_{4})=\left[\begin{smallmatrix} \mathcal{C}_{K}(\ast,\ast)  & 0 \\ 
M(U_{4}) &  \mathcal{U}(U_{1},U_{4})
\end{smallmatrix}\right]$$ is defined as follows: for $\left[\begin{smallmatrix} \gamma  & 0 \\ 
n &  \theta
\end{smallmatrix}\right]\in 
\left[\begin{smallmatrix}  \mathcal{C}_{K}(\ast,\ast) & 0 \\ 
M(U_{3}) &  \mathcal{U}(U_{2},U_{3})
\end{smallmatrix}\right]$ we set
$$\Lambda(f^{op}\otimes g)\Big(\left[\begin{smallmatrix} \gamma  & 0 \\ 
n &  \theta
\end{smallmatrix}\right]\Big)=g\circ \left[\begin{smallmatrix} \gamma  & 0 \\ 
n &  \theta
\end{smallmatrix}\right]\circ f=\left[\begin{smallmatrix} \lambda'  & 0 \\ 
m' &  \beta'
\end{smallmatrix}\right] \circ \left[\begin{smallmatrix} \gamma  & 0 \\ 
n &  \theta
\end{smallmatrix}\right]\circ \left[\begin{smallmatrix} \lambda  & 0 \\ 
m &  \beta
\end{smallmatrix}\right].$$
Hence, for $\left[\begin{smallmatrix} \gamma  & 0 \\ 
n &  0
\end{smallmatrix}\right]\in 
\left[\begin{smallmatrix}  \mathcal{C}_{K}(\ast,\ast) & 0 \\ 
M(U_{3}) & 0
\end{smallmatrix}\right]=\mathcal{I}(\mathfrak{M}_{2},\mathfrak{M}_{3})$ we have that

$$\mathcal{I}(f^{op}\otimes g)\Big(\left[\begin{smallmatrix} \gamma  & 0 \\ 
n &  0
\end{smallmatrix}\right]\Big)=\left[\begin{smallmatrix} \lambda'  & 0 \\ 
m' &  \beta'
\end{smallmatrix}\right] \circ \left[\begin{smallmatrix} \gamma  & 0 \\ 
n &  0
\end{smallmatrix}\right]\circ \left[\begin{smallmatrix} \lambda  & 0 \\ 
m &  \beta
\end{smallmatrix}\right]=\left[\begin{smallmatrix} \lambda'  & 0 \\ 
m' &  \beta'
\end{smallmatrix}\right] \circ \left[\begin{smallmatrix}\gamma \lambda  & 0 \\ 
n\lambda &  0
\end{smallmatrix}\right]=\left[\begin{smallmatrix}\lambda'\gamma \lambda  & 0 \\ 
m'\bullet(\gamma\lambda)+\beta'\bullet (n\lambda) &  0
\end{smallmatrix}\right].$$
Now, we consider the following object $\mathfrak{N}= \left[\begin{smallmatrix} \ast & 0 \\ 
\underline{M} &  0
\end{smallmatrix}\right]\in \Lambda$ and
$$\Lambda^{e}\Big((\mathfrak{N},\mathfrak{N}),(-,-)\Big)\in \mathrm{Mod}(\Lambda^{e}).$$
 For objects $\mathfrak{M}=\left[\begin{smallmatrix}
\ast & 0 \\ 
\underline{M} & U
\end{smallmatrix}\right]$ and $\mathfrak{M}'=\left[\begin{smallmatrix}
\ast & 0 \\ 
\underline{M} & U'
\end{smallmatrix}\right]$ in $\Lambda$ we have that

$$\Lambda^{e}\Big((\mathfrak{N},\mathfrak{N}),(\mathfrak{M},\mathfrak{M}')\Big)=\Lambda(\mathfrak{M},\mathfrak{N})\otimes_K \Lambda(\mathfrak{N},\mathfrak{M}')=\left[\begin{smallmatrix}
K & 0 \\ 
0 & 0
\end{smallmatrix}\right]\otimes_K  \left[\begin{smallmatrix}
K & 0 \\ 
M(U') & 0
\end{smallmatrix}\right]$$
For $$f:=\left[\begin{smallmatrix}
\lambda & 0 \\ 
m & \beta
\end{smallmatrix}\right]:\left[\begin{smallmatrix}
 \ast & 0 \\ 
\underline{M} & U_{1}
\end{smallmatrix}\right]=\mathfrak{M}_{1}\longrightarrow \left[\begin{smallmatrix}
\ast & 0 \\ 
\underline{M} & U_{2}
\end{smallmatrix}\right]=\mathfrak{M}_{2}, \text{ and}$$ 
$$g:=\left[\begin{smallmatrix}
\lambda' & 0 \\ 
m' & \beta'
\end{smallmatrix}\right]:\left[\begin{smallmatrix}
 \ast & 0 \\ 
\underline{M} & U_{3}
\end{smallmatrix}\right]=\mathfrak{M}_{3}\longrightarrow \left[\begin{smallmatrix}
\ast & 0 \\ 
\underline{M} & U_{4}
\end{smallmatrix}\right]=\mathfrak{M}_{4}$$ we have $f^{op}\otimes g:(\mathfrak{M}_{2},\mathfrak{M}_{3})\longrightarrow (\mathfrak{M}_{1},\mathfrak{M}_{4})$ a morphism in $\Lambda^{e}$ and hence we have a morphism of abelian groups

$$\Lambda^{e}\Big((\mathfrak{N},\mathfrak{N}),(f^{op}\otimes g)\Big):\Lambda^{e}\Big((\mathfrak{N},\mathfrak{N}),(\mathfrak{M}_{2},\mathfrak{M}_{3})\Big)\longrightarrow \Lambda^{e}\Big((\mathfrak{N},\mathfrak{N}),(\mathfrak{M}_{1},\mathfrak{M}_{4})\Big).$$
In this case, we have that
\begin{align*}
\Lambda^{e}\Big((\mathfrak{N},\mathfrak{N}),(\mathfrak{M}_{2},\mathfrak{M}_{3})\Big)=\Lambda^{op}(\mathfrak{N},\mathfrak{M}_{2})\otimes_K \Lambda(\mathfrak{N},\mathfrak{M}_{3}) & =\Lambda(\mathfrak{M}_{2},\mathfrak{N})\otimes_K \Lambda(\mathfrak{N},\mathfrak{M}_{3})\\
& = \left[\begin{smallmatrix}
K & 0 \\ 
0 & 0
\end{smallmatrix}\right]\otimes_K  \left[\begin{smallmatrix}
K & 0 \\ 
M(U_{3}) & 0
\end{smallmatrix}\right].
\end{align*}
Therefore,  for  $\left[\begin{smallmatrix}
\gamma & 0 \\ 
0 & 0
\end{smallmatrix}\right]\otimes  \left[\begin{smallmatrix}
\delta & 0 \\ 
n & 0
\end{smallmatrix}\right]\in  \left[\begin{smallmatrix}
K & 0 \\ 
0 & 0
\end{smallmatrix}\right]\otimes_K  \left[\begin{smallmatrix}
K & 0 \\ 
M(U_{3}) & 0
\end{smallmatrix}\right]$ we have that

\begin{align*}
 \Lambda^{e}\Big((\mathfrak{N},\mathfrak{N}),(f^{op}\otimes g)\Big)\Big(\left[\begin{smallmatrix}
\gamma & 0 \\ 
0 & 0
\end{smallmatrix}\right]\otimes  \left[\begin{smallmatrix}
\delta & 0 \\ 
n & 0
\end{smallmatrix}\right]\Big) & =(f^{op}\otimes g)\circ \Big(\left[\begin{smallmatrix}
\gamma & 0 \\ 
0 & 0
\end{smallmatrix}\right]\otimes  \left[\begin{smallmatrix}
\delta & 0 \\ 
n & 0
\end{smallmatrix}\right]\Big)\\
& =\Big(\left[\begin{smallmatrix}
\gamma & 0 \\ 
0 & 0
\end{smallmatrix}\right]\circ f \Big)\otimes \Big(g\circ \left[\begin{smallmatrix}
\delta& 0 \\ 
n & 0
\end{smallmatrix}\right]\Big)\\
& = \Big(\left[\begin{smallmatrix}
\gamma & 0 \\ 
0 & 0
\end{smallmatrix}\right]\circ \left[\begin{smallmatrix}
\lambda & 0 \\ 
m & \beta
\end{smallmatrix}\right]\Big)\otimes \Big(\left[\begin{smallmatrix}
\lambda' & 0 \\ 
m' & \beta'
\end{smallmatrix}\right]\left[\begin{smallmatrix}
\delta& 0 \\ 
n & 0
\end{smallmatrix}\right]\Big)\\
& = \left[\begin{smallmatrix}
\gamma\lambda & 0 \\ 
0 & 0
\end{smallmatrix}\right]\otimes  \left[\begin{smallmatrix}
\lambda' \delta& 0 \\ 
m'\bullet \delta+\beta'\bullet n & 0
\end{smallmatrix}\right].
\end{align*}

For $\mathfrak{M}=\left[\begin{smallmatrix}
\ast & 0 \\ 
\underline{M} & U
\end{smallmatrix}\right]$ and $\mathfrak{M}'=\left[\begin{smallmatrix}
\ast & 0 \\ 
\underline{M} & U'
\end{smallmatrix}\right]$ in $\Lambda$, we consider the canonical isomorphism
$$\Phi_{\mathfrak{M},\mathfrak{M}'}:\left[\begin{smallmatrix}
K & 0 \\ 0 & 0
\end{smallmatrix}\right]\otimes \left[\begin{smallmatrix}
K & 0 \\ M(U') & 0
\end{smallmatrix}\right]\longrightarrow  \left[\begin{smallmatrix}
K & 0 \\ M(U') & 0
\end{smallmatrix}\right]$$
defined as 
$$\Phi_{\mathfrak{M},\mathfrak{M}'}\Big(\left[\begin{smallmatrix}
a & 0 \\ 0 & 0
\end{smallmatrix}\right]\otimes \left[\begin{smallmatrix}
b & 0 \\ x & 0
\end{smallmatrix}\right]\Big)= \left[\begin{smallmatrix}
ab & 0 \\ ax & 0
\end{smallmatrix}\right],$$
where $ax$ is defined using the structure of $K$-vector space on $M(U')$.\\
Hence we have the following commutative diagram
$$\xymatrix{{\left[\begin{smallmatrix}
K & 0 \\ 0 & 0
\end{smallmatrix}\right]\otimes \left[\begin{smallmatrix}
K & 0 \\ M(U_{3}) & 0
\end{smallmatrix}\right]}\ar[rr]^{\Phi_{\mathfrak{M}_{2},\mathfrak{M}_{3}}}\ar[d]_{\Lambda^{e}\Big((\mathfrak{N},\mathfrak{N}),(f^{op}\otimes g)\Big)}& &  {\left[\begin{smallmatrix}
K & 0 \\ M(U_{3}) & 0
\end{smallmatrix}\right]}\ar[d]^{\mathcal{I}(f^{op}\otimes g)}\\
{\left[\begin{smallmatrix}
K & 0 \\ 0 & 0
\end{smallmatrix}\right]\otimes \left[\begin{smallmatrix}
K & 0 \\ M(U_{4}) & 0
\end{smallmatrix}\right]}\ar[rr]_{\Phi_{\mathfrak{M}_{1},\mathfrak{M}_{4}}} & &  {\left[\begin{smallmatrix}
K & 0 \\ M(U_{4}) & 0
\end{smallmatrix}\right]}}$$
Hence, $\mathcal{I}(-,-)\simeq\Lambda^{e}\Big((\mathfrak{N},\mathfrak{N}),(-,-)\Big)$.
\end{proof}

\begin{corollary}\label{singularonepoint}
Let $\mathcal{U}$ be a $K$-category and $M:\mathcal{U}\longrightarrow \mathrm{Mod}(K)$ a $K$-functor. Consider the one-point extension category $\Lambda:=\left[ \begin{smallmatrix}
\mathcal{C}_{K} & 0 \\ 
\underline{M} & \mathcal{U}
\end{smallmatrix}\right].$ Then there exists an equivalence of triangulated categories
$$\mathbf{D}_{sg}(\mathrm{Mod}(\Lambda))\simeq \mathbf{D}_{sg}(\mathrm{Mod}(\mathcal{U})).$$
\end{corollary}
\begin{proof}
By Lemma \ref{proyepihomomat}, we have that $\Gamma(\Phi):\Lambda\longrightarrow \mathcal{U}(-,-)\circ \Phi^{op}\otimes \Phi$ is an homological epimorphism. By Lemma \ref{bimoduloproy}, we have that $\mathcal{I}=\mathrm{Ker}(\Gamma(\Phi))$ is projective in $\mathrm{Mod}(\Lambda^{e})$. By  Theorem \ref{teoremaprincipal}, we conclude that 
$\mathbf{D}_{sg}(\mathrm{Mod}(\Lambda))\simeq \mathbf{D}_{sg}(\mathrm{Mod}(\mathcal{U})).$
\end{proof}

In order to give an explicit example we recall the following notions.

\subsection{Quivers, path algebras and path categories}  
A quiver $\Delta$ consists of a set of vertices $\Delta_{0}$  and a set of arrows $\Delta_{1}$ which is the disjoint union of sets 
$\Delta(x,y)$, where the elements of $\Delta(x,y)$ are the arrows $\alpha:x\rightarrow y$ from the vertex $x$ to the vertex $y$. Given a quiver $\Delta$, its path category $\mathrm{Pth}\Delta$ has as objects the vertices of $\Delta$  and the morphisms $x\rightarrow y$ are paths  from $x$ to $y$ which are by definition the formal compositions $\alpha_{n}\cdots\alpha_{1}$ where $\alpha_{1}$ starts in $x$, $\alpha_{n}$  ends in $y$ and the end point of $\alpha_{i}$ coincides with the start point of $\alpha_{i+1}$ for all $i\in\{1,\ldots,n-1\}$. The positive integer $n$ is called the length of the path. There is a path $\xi_{x}$ of length $0$ for each vertex to itself. The composition in $\mathrm{Pth}\Delta$ of paths of positive length is just concatenations whereas the $\xi_{x}$ act as identities.

Given a quiver $\Delta$  and a field $K$, an additive  $K$-category $K\Delta$ is associated to $\Delta$ by taking as the indecomposable objects in $K\Delta$  the vertices of $\Delta$ and hence an arbitrary object of $K\Delta $ is a direct sum of indecomposable objects. Given $x,y\in\Delta_{0}$  the set of maps from $x$ to $y$ is given by the  $K$-vector space with basis the set of all paths from $x$ to $y$. The composition in $K\Delta$ is of course obtained by $K$-linear extension of the composition in $\mathrm{Pth} \Delta$, that is, the product of two composable paths is defined to be the corresponding composition, the product of two non-composable paths is, by definition, zero. In this way we obtain an associative $K$-algebra which has unit element if and only if $\Delta_{0}$ is finite (the unit element is given by $\sum _{x\in \Delta_{0}}\xi_{x}$).\\
In $K\Delta$, we denote by $K\Delta^{+}$ the ideal generated by all arrows and by $(K\Delta^{+})^{n}$ the ideal generated by all paths of length $\ge n$.\\
Given vertices $x,y\in\Delta_{0}$, a finite linear combination $\sum_{w}\lambda_{w}w$, where $\lambda_{w}\in K$ and $w$ are paths of length $\ge 2$ from $x$ to $y$, is called a relation on $\Delta$. It can be seen that any ideal $I\subset (K\Delta^{+})^{2}$ can be generated, as an ideal, by relations. If  $I$ is generated as an ideal by the set $\{\rho_{i}\mid i\}$ of relations, we write $I=\langle \rho_{i}\mid i\rangle$.\\
Given a quiver $\Delta=(\Delta_{0},\Delta_{1})$, a representation $V=(V_{x},f_{\alpha})$ of $\Delta$ over $K$ is given by vector spaces $V_{x}$ for all $x\in \Delta_{0}$, and linear maps $f_{\alpha}:V_{x}\rightarrow V_{y}$, for any arrow $\alpha:x\rightarrow y$.   The category of representations of $\Delta$ is the category with objects the representations, and a morphism of representations  $h=(h_{x}): V\rightarrow V'$ is given by maps $h_{x}:V_{x}\rightarrow V'_{x}$ $(x\in\Delta_{0})$ such that $h_{y}f_{\alpha}=f_{\alpha'}h_{x}$ for any $\alpha:x\rightarrow y$. The category of representations of $\Delta$ is denoted by $\mathrm{Rep}(\Delta)$.\\
Given a set of relations $\langle\rho_{i}\!\!\mid \! i\rangle$ of $\Delta$,  we denote by $K\Delta/\langle\rho_{i}\!\!\mid \! i\rangle$ the path category given by the quiver $\Delta$ and relations $\rho_{i}$.  The category  of functors $\mathrm{Mod}\Big(K\Delta/\langle \rho_{i}\!\!\mid \! i\rangle\Big):=\Big(K\Delta/\langle \rho_{i}\!\!\mid \! i\rangle, \mathrm{Mod}(K)\Big)$ can be identified with the representations  of $\Delta$ satisfying the relations $\rho_{i}$ which is denoted by $\mathrm{Rep}(\Delta,\{\rho_{i}\!\! \mid \! i\})$,  (see \cite[p. 42]{RingelTame}).\\
Consider a field $K$  and the infinite quiver 
$$Q:\xymatrix{1\ar[r]^{\alpha_{1}}& 2\ar[r]^{\alpha_{2}}  &\cdots \ar[r]  &  k\ar[r]^{\alpha_{k}} & k+1\ar[r]  & \cdots\ar[r] & \cdots }$$
Then we have the path $K$-category $\mathcal{U}=KQ$. Consider the left $KQ$-module $M$ given by the representation

$$M:\xymatrix{0\ar[r]^{0}& K\ar[r]^{1} & K\ar[r]^{1} & K\ar[r]^{1}  &\cdots \ar[r]  &  K\ar[r]^{1} & K\ar[r]  & \cdots\ar[r] & \cdots }$$

%$$\xymatrix{ & 0\ar[d]_{\beta}\ar@/^1pc/[d]_{\alpha}\ar@/_1pc/[d]_{\gamma}\ar@/^1pc/[dr]_{\theta}\\
%1\ar[r]^{\alpha_{1}}& 2\ar[r]^{\alpha_{2}}  & 3\ar[r] & \cdots \ar[r]  &  k\ar[r]^{\alpha_{k}} & k+1\ar[r]  & \cdots\ar[r] & \cdots }$$
Then the one-point extension category  $\Lambda:=\left[ \begin{smallmatrix}
\mathcal{C}_{K} & 0 \\ 
\underline{M} & \mathcal{U}
\end{smallmatrix}\right]$  has the following quiver

$$K\overline{Q}:\xymatrix{ & 0\ar[d]_{\beta_{1}}\ar[dr]^{\beta_{2}}\ar[drrr]^{\beta_{k-1}}\ar[drrrr]^{\beta_{k}}\\
1\ar[r]_{\alpha_{1}}& 2\ar[r]_{\alpha_{2}}  & 3\ar[r] & \cdots \ar[r]  &  k\ar[r]_{\alpha_{k}} & k+1\ar[r]  & \cdots\ar[r] & \cdots }$$
where there is an arrow $\beta_{i}:0\longrightarrow i+1$ for all integer  $i\geq 1$ and
with relations $R=\{\alpha_{i+1}\beta_{i}-\beta_{i+1}\}_{i\geq 1}$.
Hence in this case we have an equivalence of triangulated categories
$$\mathbf{D}_{sg}\Big(\mathrm{Mod}(K\overline{Q}/R)\Big)\simeq \mathbf{D}_{sg}\Big(\mathrm{Mod}(KQ)\Big).$$

%\begin{remark}
%See p. 4580 y 4581 in \cite{Donmez}, para ver como es el quiver y las relaciones de la algebra extension por un punto.
%\end{remark}

\section*{Acknowledgements}
The authors thank to the project PAPIIT-Universidad Nacional Aut\'onoma de M\'exico IN100124.
%The authors are grateful to the project PAPIIT-Universidad Nacional Aut\'onoma de M\'exico IN100520.  

\bibliographystyle{amsplain}

\end{document}